\documentclass[11pt]{article}

\usepackage{amssymb, amsfonts, amsmath, amsthm, graphics}
\usepackage{lineno}

\newtheorem{thm}{Theorem}[section]
\newtheorem{lemma}[thm]{Lemma}
\newtheorem{prop}[thm]{Proposition}
\newtheorem{cor}[thm]{Corollary}

\theoremstyle{definition}
\newtheorem{defn}[thm]{Definition}
\newtheorem{rem}[thm]{Remark}
\newtheorem{remark}[thm]{Remark}
\newtheorem{notation}[thm]{Notation}

\newcommand{\bN}{ {\mathbb N} }
\newcommand{\cP}{{\cal P}}
\newcommand{\cO}{\Omega}
\newcommand{\Otilda}{ \widetilde{\Omega} }
\newcommand{\ob}{ {\mathcal O}_{nc}^{B} }
\newcommand{\cS}{{\cal S}}
\newcommand{\bZ}{ {\mathbb Z} }
\newcommand{\induc}{\downarrow}

\newcommand{\lgb}{  \ell_B }
\newcommand{\lgd}{  \ell_D }

\newcommand{\snc}{ {\cal S}_{nc} }
\newcommand{\snca}{ {\cal S}^{A}_{nc} }
\newcommand{\sncb}{ {\cal S}^{B}_{nc} }
\newcommand{\sncd}{ {\cal S}^{D}_{nc} }
\newcommand{\ncb}{ NC^{B} }
\newcommand{\ncd}{ NC^{D} }

\newcommand{\ecdef}{ \stackrel{\mathrm{def}}{\Longleftrightarrow} }

\begin{document}

\begin{center}
{\bf\Large Posets of annular non-crossing partitions of types B and
D}

\vspace{14pt}

Alexandru Nica \footnote{Research supported by a Discovery Grant
from NSERC, Canada, and by a PREA award from the Province of
Ontario.} \hspace{2cm} Ion Oancea

\vspace{10pt}

Department of Pure Mathematics, University of Waterloo
\end{center}

\vspace{10pt}

\begin{abstract}
We study the set $\sncb (p,q)$ of annular non-crossing permutations
of type B, and we introduce a corresponding set $\ncb (p,q)$ of
annular non-crossing partitions of type B, where $p$ and $q$ are two
positive integers. We prove that the natural bijection between
$\sncb (p,q)$ and $\ncb (p,q)$ is a poset isomorphism, where the
partial order on $\sncb (p,q)$ is induced from the hyperoctahedral
group $B_{p+q}$, while $\ncb (p,q)$ is partially ordered by reverse
refinement. In the case when $q=1$, we prove that $\ncb (p,1)$ is a
lattice with respect to reverse refinement order.

We point out that an analogous development can be pursued in type D,
where one gets a canonical isomorphism between $\sncd (p,q)$ and
$\ncd (p,q)$. For $q=1$, the poset $\ncd (p,1)$ coincides with a
poset ``$NC^{(D)} (p+1)$'' constructed in a paper by Athanasiadis
and Reiner in 2004, and is a lattice by the results of that paper.
\end{abstract}

\vspace{10pt}

\begin{center}
{\bf\large 1. Introduction}
\end{center}
\setcounter{section}{1}

\noindent Let $p$ and $q$ be two positive integers. Denote $p+q
=:n$, and consider the hyperoctahedral group $B_n$ -- that is, the
group of permutations $\tau$ of $\{ 1, \ldots , n \} \cup \{ -1,
\ldots , -n \}$ with the property that $\tau (-i) = - \tau (i)$ for
every $1 \leq i \leq n$. We will use the notation $\sncb (p,q)$ for
the set of permutations $\tau \in B_n$ that can be drawn without
crossings (in a sense explained precisely in subsection 2.5 and in
Definition \ref{def:3.1} below) inside an annulus which has the
points $1, \ldots , p, -1, \ldots , -p$ marked clockwisely on its
outer circle, and has the points $p+1, \ldots , n, -(p+1), \ldots ,
-n$ marked counterclockwisely on its inner circle. A concrete
example of drawing of a permutation $\tau \in \sncb (p,q)$ is shown
in Figure 1.

\begin{center}
\scalebox{.32}{\includegraphics{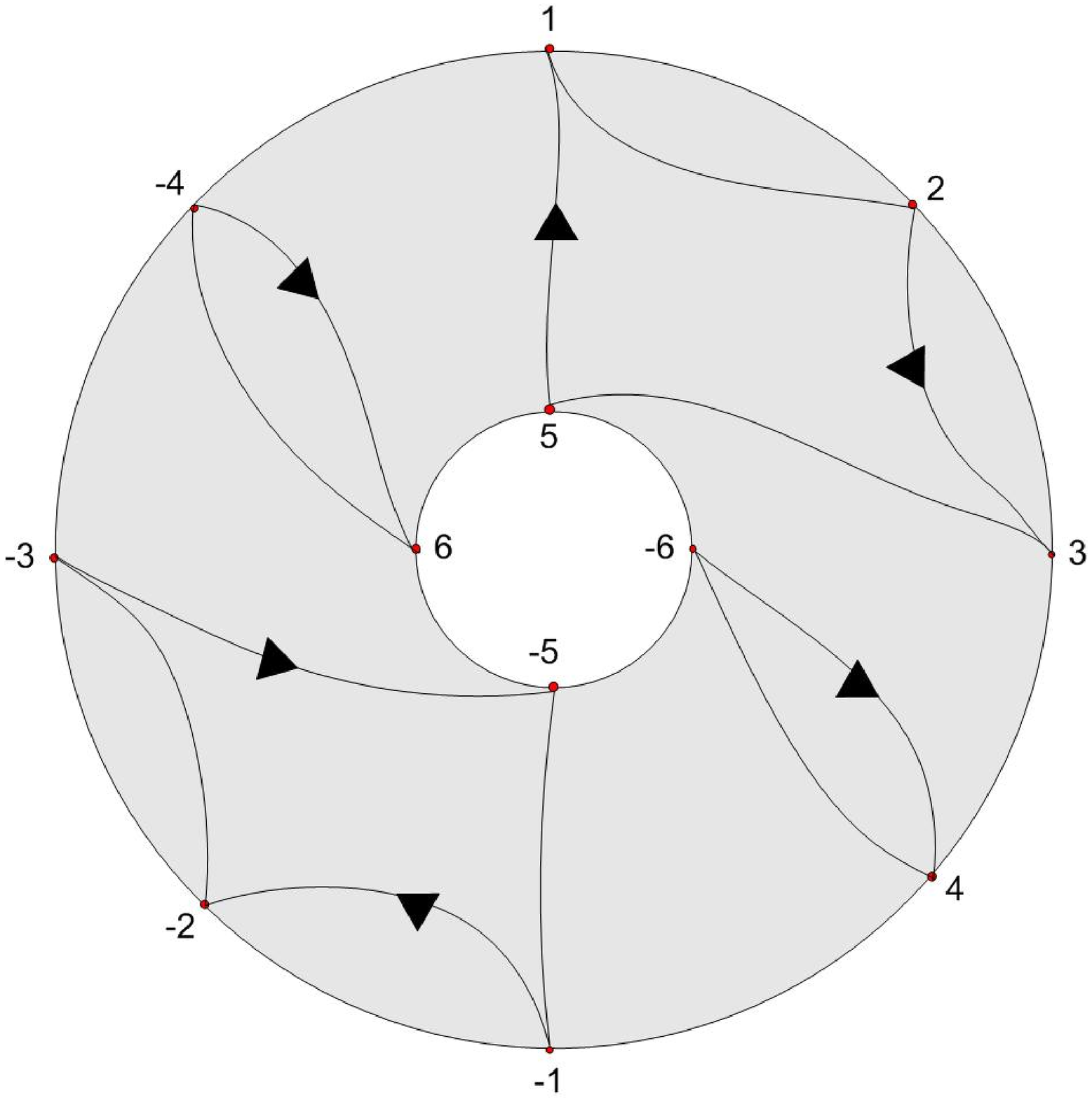}}

{\bf Figure 1.} An example of annular non-crossing permutation of
type B:

$\tau = (1,2,3,5)(4,-6)(-1,-2,-3,-5)(-4,6) \in \sncb (4,2)$.
\end{center}

In recent research literature started by \cite{Be03}, \cite{BW02}
one considers a length function $\lgb : B_n \to \bN \cup \{ 0 \}$
which is invariant under conjugation, and a partial order on $B_n$
defined by the condition that
\begin{equation}  \label{eqn:1.01}
\sigma \leq \tau \mbox{ in $B_n$} \ \ecdef \ \lgb ( \tau ) = \lgb (
\sigma ) + \lgb ( \sigma^{-1} \tau ), \ \ \sigma , \tau \in B_n.
\end{equation}
The first result of the present paper is stated as follows.

\begin{thm}  \label{theorem1}
In the notations introduced above, we have that
\begin{equation}  \label{eqn:1.02}
\sncb (p,q) = \{ \tau \in B_n \mid \tau \leq \gamma \},
\end{equation}
where $\gamma = (1, \ldots , p,-1, \ldots , -p)(p+1, \ldots , n,
-(p+1), \ldots , -n) \in B_n$.
\end{thm}

\begin{notation}   \label{def:1.2}
$1^o$ For a permutation $\tau \in B_n$ we will denote by $\Omega (
\tau )$ the partition of $\{ 1, \ldots , n \} \cup \{ -1, \ldots ,
-n \}$ into cycles of $\tau$. If $A$ is a block of $\Omega ( \tau )$
then, clearly, the set $-A := \{ -a \mid a \in A \}$ is a block of
$\Omega ( \tau )$ as well. We have that either $A \cap ( -A ) =
\emptyset$ or $A = -A$; in the latter case we say that $A$ is an
{\bf inversion-invariant block}, or that it is a {\bf zero-block} of
$\Omega ( \tau )$.

$2^o$ Let $\tau$ be in $B_n$, and let us write explicitly
\begin{equation}  \label{eqn:1.03}
\Omega ( \tau ) = \{ A_1, -A_1, \ldots , A_k, -A_k, Z_1, \ldots ,
Z_l \},
\end{equation}
where $Z_1, \ldots , Z_l$ ($0 \leq l \leq n$) are the zero-blocks of
$\Omega ( \tau )$. Then we denote
\begin{equation}  \label{eqn:1.04}
\Otilda ( \tau ) = \{ A_1, -A_1, \ldots , A_k, -A_k, Z_1 \cup \cdots
\cup Z_l \}
\end{equation}
(a new partition of $\{ 1, \ldots , n \} \cup \{ -1, \ldots , -n
\}$, which has at most one zero-block).
\end{notation}

In this paper we introduce the set $\ncb (p,q)$ of partitions of $\{
1, \ldots , n \} \cup \{ -1, \ldots , -n \}$, defined as follows.

\begin{defn}
In the notations set above, we put
\begin{equation}  \label{eqn:1.05}
\ncb (p,q) := \{ \Otilda ( \tau ) \mid \tau \in \sncb (p,q) \}.
\end{equation}
We view $\ncb (p,q)$ as a partially ordered set, with partial order
given by reverse refinement ($\pi \leq \rho$ if and only if every
block of $\rho$ is a union of blocks of $\pi$).
\end{defn}

\begin{thm} \label{theorem2}
The function
\begin{equation}  \label{eqn:1.06}
\sncb (p,q) \ni \tau \mapsto \Otilda ( \tau ) \in \ncb (p,q)
\end{equation}
is bijective, and is moreover a poset isomorphism, where the partial
order on $\sncb (p,q)$ is the one induced from $B_n$ (as in Equation
(\ref{eqn:1.01})), while $\ncb (p,q)$ is partially ordered by
reverse refinement.
\end{thm}

The fact that $\Otilda ( \tau )$ (rather than $\Omega ( \tau )$
itself) is used in Theorem \ref{theorem2} comes from an
order-preservation issue. The function $\tau \mapsto \Omega ( \tau
)$ is one-to-one on $\sncb (p,q)$ (see Remark \ref{rem:4.6} below),
but is not order-preserving -- it is immediate, for instance, that
there exist permutations $\tau \in \sncb (p,q)$ such that $\cO (
\tau ) \not\leq \cO ( \gamma )$ (even though Theorem \ref{theorem1}
asserts that $\tau \leq \gamma$ for every $\tau \in \sncb (p,q)$).
The adjustment from $\Omega ( \tau )$ to $\Otilda ( \tau )$ corrects
this problem.

It is natural to ask whether $\ncb (p,q)$ is a lattice under the
reverse refinement order. This is equivalent, by Theorem
\ref{theorem2}, to asking if $\sncb (p,q)$ is a lattice with respect
to the partial order inherited from $B_n$. It turns out that $\ncb
(p,q)$ is not a lattice when $p,q \geq 2$; but it is nevertheless
interesting to see that that the following holds:

\begin{thm}  \label{theorem3}
For $n \geq 2$, the poset $\ncb (n-1, 1)$ is a lattice. The meet
operation on $\ncb (n-1,1)$ is the restriction of the meet operation
on the lattice of all partitions of $\{ 1, \ldots , n \} \cup \{ -1,
\ldots , -n \}$; that is, for $\pi , \rho \in \ncb (n-1,1)$, the
blocks of the meet $\pi \wedge \rho  \in \ncb (n-1,1)$ are precisely
the non-empty intersections $A \cap B$ where $A$ is a block of $\pi$
and $B$ is a block of $\rho$.
\end{thm}

\begin{remark}  \label{rem:1.6}
The theorems presented above refer to the combination of two
frameworks for studying non-crossing permutations and partitions
that have appeared (separately from each other) in the recent
research literature. In this remark we comment briefly on how the
results of the present paper are (or are not) analogous to known
results holding in these two separate frameworks.

\vspace{6pt}

{\em Framework I: non-crossing permutations of type B in the disc.}

Theorems \ref{theorem1} and \ref{theorem2} are faithful analogues
for results known to hold for non-crossing permutations and
partitions of type B that are drawn in a disc (rather than in an
annulus). Here partitions were considered before permutations, in
the work of Reiner \cite{R97}. The poset $\ncb (n)$ of (disc)
non-crossing partitions of type B consists of those partitions $\pi$
of $\{ 1, \ldots n \} \cup \{ -1, \ldots , -n \}$ which are
non-crossing with respect to the order $1<2< \cdots <n <-1 < -2 <
\cdots < -n$, and have the symmetry property that if $A$ is a block
of $\pi$ then $-A$ is a block of $\pi$ as well. $\ncb (n)$ embeds
naturally into the hyperoctahedral group $B_n$, and one can define
$\sncb (n)$ as the image of $\ncb (n)$ under this embedding. The
inverse of the canonical bijection $\ncb (n) \mapsto \sncb (n)$ is
precisely the restriction to $\sncb (n)$ of the orbit map $\tau
\mapsto \Omega ( \tau )$ from Notation \ref{def:1.2}.1. It turns out
(see Theorem 4.9 of \cite{BW02}, or Section 4.2 of \cite{Be03}, or
Theorem 3.2 of \cite{BGN03}) that
\begin{equation}  \label{eqn:1.07}
\sncb (n) =  \{ \tau \in B_n \mid \tau \leq \gamma_o \} ,
\end{equation}
where $\gamma_o = (1,2, \ldots,n, -1, -2, \ldots , -n)$ and where
the partial order considered on $B_n$ is the same as above (defined
by the formula (\ref{eqn:1.01})). Moreover, the bijection
\begin{equation}  \label{eqn:1.08}
\sncb (n) \ni \tau \mapsto \Omega ( \tau ) \in \ncb (n)
\end{equation}
is a poset isomorphism, where $\sncb (n)$ is considered with the
partial order from (\ref{eqn:1.01}) while $\ncb (n)$ is partially
ordered by reverse refinement. Theorems \ref{theorem1} and
\ref{theorem2} can be viewed as annular counterparts for these facts
known from the disc case.

\vspace{6pt}

{\em Framework II: annular non-crossing permutations of type A.}

Here we consider the set $\snca (p,q)$ of permutations $\tau$ of $\{
1, \ldots , p+q \}$ that can be drawn without crossings inside an
annulus which has the points $1, \ldots , p$ marked clockwisely on
its outer circle and has the points $p+1, \ldots , p+q$ marked
counterclockwisely on its inner circle. (Unlike in type B, there are
no additional symmetry requirements that $\tau$ has to satisfy.) It
is intriguing that the above Theorems \ref{theorem1} and
\ref{theorem2} {\em are not} counterparts of type B for some
theorems that hold for $\snca (p,q)$. Indeed, the relation between
$\snca (p,q)$ and the poset of partitions of $\{ 1, \ldots , p+q \}$
is marred by the fact that the orbit map $\tau \mapsto \Omega ( \tau
)$ {\em is not one-to-one} on $\snca (p,q)$ (see Section 4 of
\cite{MN04} for a detailed discussion of why this happens). On the
other hand it is easily seen that $\snca (p,q)$ is not an interval
with respect to the natural partial order (analogous to the one from
formula (\ref{eqn:1.01})) that one can define on the group of all
permutations of $\{ 1, \ldots , p+q  \}$. Thus annular non-crossing
permutations of type A don't relate so well to posets of
set-partitions. From this perspective, the goal of the present paper
is to show that the situation improves by quite a bit when one adds
symmetry requirements of type B.
\end{remark}

\begin{remark}
All three theorems presented above also have analogues living in the
framework of Weyl groups of type D. We discuss these analogues in
Section 7 of the paper. For Theorems \ref{theorem1} and
\ref{theorem2}, the corresponding facts about $\sncd (p,q)$ and
$\ncd (p,q)$ are easily derived out of their counterparts of type B
(see Corollaries \ref{cor:7.1} and \ref{cor:7.2} below). Concerning
the type D counterpart for Theorem \ref{theorem3}, it turns out that
$\ncd (n-1,1)$ coincides exactly with the poset ``$NC^{(D)} (n)$''
constructed in the paper \cite{AR04} by Athanasiadis and Reiner, and
is hence a lattice by the results of that paper.
\end{remark}

\begin{remark}
Since introducing the symmetry of type B improves the situation and
leads to nicer posets of annular non-crossing partitions, it is of
clear interest to look at the enumerative properties of these newly
introduced structures. Some results in this direction are obtained
in \cite{GNO}, where the rank-generating function and the M\"obius
function of $\ncb (p,q)$ are studied.
\end{remark}

\begin{remark} \label{rem:1.10}
{\em (Organization of the paper.)} Besides the introduction section,
the paper has six other sections. Section 2 contains a review of
some background and notations. In Section 3 we prove Theorem
\ref{theorem1}, then Section 4 is devoted to discussing the map
$\Otilda$ and to proving Theorem \ref{theorem2}. The proof of
Theorem \ref{theorem3} is divided between the Sections 5 and 6 of
the paper. Section 5 still uses the framework of $\ncb (p,q)$ where
$p,q$ are arbitrary positive integers. We study intersection meets
of partitions from $\ncb (p,q)$, and find out there is only one
possibility for how it can happen that $\pi, \rho \in \ncb (p,q)$,
but the intersection meet $\pi \wedge \rho$ is no longer in $\ncb
(p,q)$: a certain permutation canonically associated to $\pi \wedge
\rho$ must display an annular crossing pattern called ``(AC-3)''
(see Remark \ref{remark5.11} below). In Section 6 we observe that
this unpleasant phenomenon can only take place when both $p$ and $q$
are at least equal to 2, and this gives us the proof of Theorem
\ref{theorem3}. Finally, Section 7 discusses the type D analogues
for the results presented above in type B.
\end{remark}

$\ $

\begin{center}
{\bf\large 2. Background and notations}
\end{center}
\setcounter{section}{2} \setcounter{equation}{0} \setcounter{thm}{0}

\subsection{Some general notations}

For a finite set $X$ we will denote by $\cP (X)$ the set of all
partitions of $X$, and we will denote by $\cS (X)$ the set of all
permutations of $X$.  If $\tau  \in \cS (X)$, then the action of
$\tau$ splits $X$ into {\bf orbits} of $\tau$ (where $x,y \in X$ are
in the same orbit of $\tau$ if and only if there exists $m \in \bZ$
such that $\tau^m (x) =y$). The number of orbits of $\tau$ will be
denoted by $\# ( \tau )$. As already mentioned in Notation
\ref{def:1.2}, the partition of $X$ into orbits of $\tau$ will be
denoted by $\cO ( \tau )$.

Another notation used throughout the paper concerns the concept of
``permutation induced by $\tau \in \cS (X)$ on a subset $A$ of $X$''
(which makes sense even if $A$ is not invariant under the action of
$\tau$). The definition for this goes as follows.

\begin{defn}   \label{def:2.1}
Let $X$ be a finite set, let $\tau$ be a permutation of $X$, and let
$A$ be a non-empty subset of $X$. The {\bf permutation of $A$
induced by $\tau$} will be denoted by $\tau \induc A$, and is the
permutation in $\cS (A)$ defined as follows: for every $a \in A$ we
look at the sequence (of elements of $X$) $\tau (a), \tau^2 (a),
\ldots , \tau^k (a), \ldots$ and define $( \tau \induc A)  (a)$ to
be the first element of this sequence which is again in $A$.
\end{defn}

\subsection{Length-function and partial order on the group $B_n$}

The length function $\lgb :B_n \to \bN \cup \{ 0 \}$ used in
Equation (\ref{eqn:1.01}) of the introduction is defined in terms of
the following set of generators for $B_n$:
\[
\{ (i,j)(-i,-j) \mid 1 \leq i<j \leq n \} \cup \{ (i,-j)(-i,j) \mid
1 \leq i<j \leq n \}
\]
\begin{equation}  \label{eqn:2.01}
\cup \{ (i, -i) \mid 1 \leq i \leq n \} .
\end{equation}
More precisely: for every $\tau \in B_n$, the length $\lgb ( \tau )$
is defined as the smallest possible $k$ such that $\tau$ can be
factored as a product of $k$ generators from (\ref{eqn:2.01}) (with
the convention that a product of 0 generators gives the unit of
$B_n$). It is easily verified that the length $\lgb$ can be
equivalently defined by the formula
\begin{equation}\label{eqn:2.02}
\lgb ( \tau ) = n-m,
\end{equation}
where $m$ is the number of pairs of non-inversion-invariant orbits
of $\tau \in B_m$.

By starting from the length function $\lgb$, one introduces a
partial order relation on $B_n$, in the way described in Equation
(\ref{eqn:1.01}) of the introduction. Later in the paper we will
need to use the explicit description for covers with respect to this
partial order. (Given $\sigma , \tau \in B_n$, recall that $\tau$ is
said to {\bf cover} $\sigma$ when $\sigma < \tau$ and there exists
no $\phi \in B_n$ such that $\sigma < \phi  < \tau$.) This goes as
follows.

\begin{prop}   \label{proposition3.5}
Let $\sigma$ and $\tau$ be two permutations in $B_n$. Then $\tau$
covers $\sigma$ if and only if one of the following four situations
takes place.

(a) $\sigma^{-1} \tau$ is of the form $(i, -i)$, where $i$ and $-i$
belong to different orbits of $\sigma$.

(b) $\sigma^{-1} \tau$ is of the form $(i,j)(-i,-j)$ with $|i| \neq
|j|$, where $i$ and $-i$ belong to the same orbit of $\sigma$, but
$j$ and $-j$ do not belong to the same orbit of $\sigma$.

(c) $\sigma^{-1} \tau$ is of the form $(i,j)(-i,-j)$ with $|i| \neq
|j|$, where no two of $i, -i, j, -j$ belong to the same orbit of
$\sigma$.

(d) $\sigma^{-1} \tau$ is of the form $(i,j)(-i,-j)$ with $|i| \neq
|j|$, where $i$ and $-j$ belong to the same orbit of $\sigma$, and
this orbit is not inversion-invariant (hence does not contain $-i$
and $j$).
\end{prop}

For a proof of Proposition \ref{proposition3.5}, see for instance
Section 3 of \cite{BW02}.

\subsection{Non-crossing permutations}

Let $\tau$ and $\gamma$ be permutations of a finite set $X$. Besides
the numbers $\# ( \tau )$ and $\# ( \gamma )$ (that count the orbits
of $\tau$ and respectively of $\gamma$) let us also consider the
number $\# ( \tau , \gamma )$ which counts the orbits for the action
on $X$ of the subgroup of $\cS (X)$ generated by $\{ \tau , \gamma
\}$. The {\bf genus formula} for $\tau$ and $\gamma$ says that the
quantity $g$ defined by
\begin{equation}\label{eqn:2.03}
\Bigl(  |X| + 2 \cdot \# ( \tau , \gamma ) \Bigr) - \Bigl( \# ( \tau
) + \# ( \tau^{-1} \gamma ) + \# ( \gamma ) \Bigr) \ = \ 2g
\end{equation}
has to be a non-negative integer. The significance of $g$ is as of
genus for a certain orientable surface constructed from $\tau$ and
$\gamma$. Formula (\ref{eqn:2.03}) goes back at least to the 1960's
(see \cite{J68}),  and appears in various forms in the literature on
factorizations of permutations (see e.g. Section 2 of \cite{GJ97}).
For a detailed exposition of the underlying theory of graphs on
surfaces see Chapter 1 of \cite{LZ04} (where the above formula can
be found in Section 1.5, Proposition 1.5.3).

In this paper we will reserve the name ``non-crossing'' for the
situation when $g=0$, that is, for the situation when the
non-crossing drawings for $\tau$ and $\gamma$ are made in the plane.
In (\ref{eqn:2.03}) we fix $\gamma$ as our ``reference
permutation'', and we make the following definition.

\begin{defn} \label{def:2.3}
Let $X$ be a finite set and let $\gamma$ be a permutation of $X$.
The set of {\bf non-crossing permutations of $X$ with respect to
$\gamma$} is
\begin{equation} \label{eqn:2.4}
\snc (X, \gamma ) := \Bigl\{ \tau \in \cS (X) \ \vline \ \# ( \tau )
+ \# ( \tau^{-1} \gamma ) + \# ( \gamma ) = |X| + 2 \cdot \# ( \tau
, \gamma )  \Bigr\} .
\end{equation}
\end{defn}

In other words, what we do is to start with a planar picture where
the elements of $X$ are represented as connected by the cycles of
$\gamma$; then $\snc (X, \gamma )$ consists of those permutations
$\tau \in \cS (X)$ which can be drawn without crossings in this
picture. In this paper we are dealing with the situations when $\# (
\gamma ) = 1$ and when $\# ( \gamma ) = 2$. These situations are
discussed in more detail and are illustrated with pictures in the
next two subsections.

\subsection{ \boldmath{$\snc (X, \gamma )$} in the case when
\boldmath{$\# ( \gamma ) = 1$} }

If $\# ( \gamma ) =1$ then $\# ( \gamma , \tau ) =1$ for every $\tau
\in \cS (X)$. The genus formula (\ref{eqn:2.03}) gives us that
\begin{equation}  \label{eqn:2.05}
\# ( \tau ) + \# ( \tau^{-1} \gamma ) \leq |X| +1 , \ \ \forall \,
\tau \in \cS (X),
\end{equation}
and Definition \ref{def:2.3} becomes
\begin{equation}  \label{eqn:2.06}
\snc (X, \gamma ) = \{ \tau \in \cS (X) \mid \# ( \tau ) + \# (
\tau^{-1} \gamma ) = |X| +1  \} .
\end{equation}
The description in (\ref{eqn:2.06}) is very useful, but is not how
one usually introduces $\snc (X, \gamma )$ in the literature on
non-crossing partitions and permutations. (For a survey of the
fairly extensive literature on this topic, see e.g. \cite{S00}.)
When $\# ( \gamma ) =1$, the usual way of introducing $\snc (X,
\gamma )$ is as the set of permutations that ``avoid the crossing
pattern $(1,3)(2,4)$''; this is precisely stated on the right-hand
side of the equivalence (\ref{eqn:2.07}) in Proposition
\ref{prop:2.5} below.

\begin{defn}   \label{def:2.4}
Let $X$ be a finite set, and let $\gamma \in \cS (X)$ be such that
$\# ( \gamma ) = 1$.

$1^o$ Let $\tau$ be a permutation of $X$. If for every orbit $A$ of
$\tau$ we have $\tau \induc A = \gamma \induc A$ (equality of
induced permutations, considered in the sense of Definition
\ref{def:2.1}), then we will say that $\tau$ is {\bf compatible
with} $\gamma$.

$2^o$ Let $\tau$ be a permutation of $X$. If there exist four
distinct elements $a,b,c,d \in X$ such that $\gamma \induc \{
a,b,c,d \} = (a,b,c,d)$ and $\tau \induc \{ a,b,c,d \} =
(a,c)(b,d)$, then we will say that $\tau$ {\bf has the crossing
pattern (DC)} with respect to $\gamma$.
\end{defn}

\begin{prop}   \label{prop:2.5}
Let $X$ be a finite set, and let $\gamma \in \cS (X)$ be such that
$\# ( \gamma ) = 1$. Consider the set of non-crossing permutations
$\snc (X, \gamma )$, defined as in Equation (\ref{eqn:2.06}). For a
permutation $\tau$ of $X$ we then have the equivalence:
\begin{equation} \label{eqn:2.07}
\tau \in \snc (X, \gamma ) \ \Leftrightarrow \ \left\{
\begin{array}{l}
\mbox{$\tau$ is compatible with $\gamma$, and $\tau$ does not}  \\
\mbox{have the crossing pattern (DC) with respect to $\gamma$.}
\end{array}  \right.
\end{equation}
\end{prop}

For a proof of Proposition \ref{prop:2.5}, see e.g. Section 1.3 of
\cite{B97}, or Section 2 of \cite{Br01}.

The initials ``DC'' in Definition \ref{def:2.4} stand for
``Disc-Crossing''. This is in relation to the fact that in order to
draw permutations in $\cS_{nc} (X, \gamma)$, one starts by
representing the elements of $X$ as points on the boundary of a
disc, in the cyclic order indicated by $\gamma$, and then the cycles
of $\tau$ are represented by drawing contours inside that disc. An
illustration of how this goes is shown in Figure 2 below.

\begin{center}
\scalebox{.32}{\includegraphics{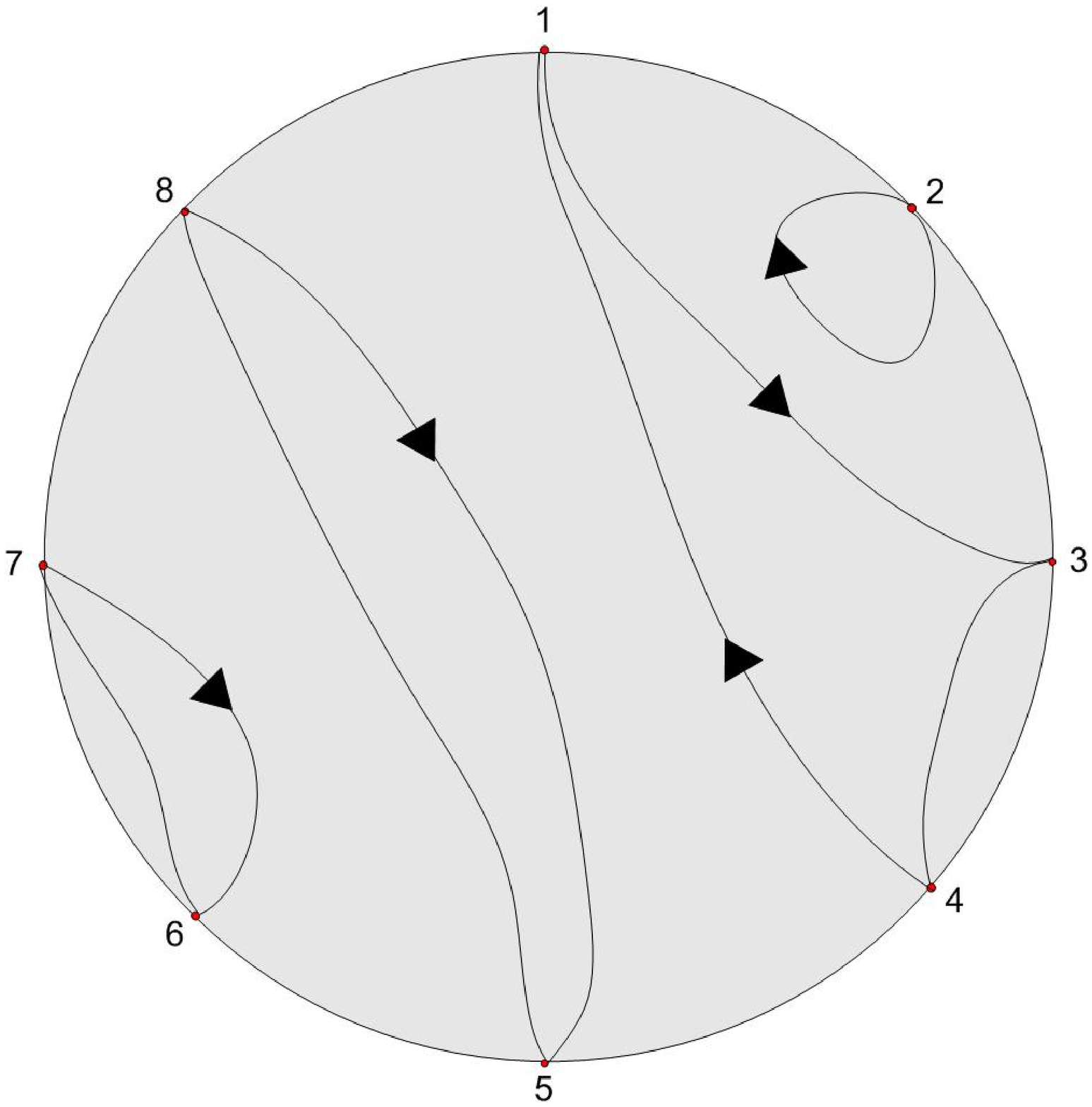}}

{\bf Figure 2.} An example of a non-crossing permutation in the
disc:

$X= \{ 1, \ldots , 8 \}, \ \gamma = (1,2, \ldots , 8), \ \tau =
(1,3,4)(2)(5,8)(6,7) \in \cS_{nc} ( X, \gamma )$.
\end{center}

\vspace{10pt}

\subsection{ \boldmath{$\snc (X, \gamma )$} in the case when
\boldmath{$\# ( \gamma ) = 2$} }

In this subsection we fix a finite set $X$ and a permutation $\gamma
\in \cS (X)$ such that $\# ( \gamma ) =2$. The two orbits of
$\gamma$ will be denoted by $Y$ and $Z$. In order to spell out what
$\snc (X, \gamma )$ is in this case, it will be convenient to use
the following defininition.

\begin{defn}  \label{definition2.11}
A subset $A \subseteq X$ such that $A \cap Y \neq \emptyset \neq A
\cap Z$ will be said to be $\gamma$-{\bf connected}. A partition
$\pi \in \cP  (X)$ will be said to be $\gamma$-{\bf connected} when
it has at least one $\gamma$-connected block, and will be said to be
$\gamma$-{\bf disconnected} in the opposite case. Finally, a
permutation $\tau \in \cS (X)$ will be said to be $\gamma$-{\bf
connected} (respectively $\gamma$-{\bf disconnected}) when the orbit
partition $\cO ( \tau )$ is so.
\end{defn}

It is clear that for $\tau \in \cS (X)$ we have
\[
\# ( \tau , \gamma ) = \left\{
\begin{array}{ccl}
1 & \mbox{ if } & \mbox{$\tau$ is $\gamma$-connected} \\
2 & \mbox{ if } & \mbox{$\tau$ is $\gamma$-disconnected.}
\end{array}  \right.
\]
The inequality provided by the genus formula thus splits in two
cases:
\begin{equation} \label{eqn:2.08}
\Bigl( \tau \in \cS (X), \mbox{$\gamma$-connected} \Bigr)
\Rightarrow \# ( \tau ) + \# ( \tau^{-1} \gamma ) \leq |X|,
\end{equation}
and
\begin{equation} \label{eqn:2.09}
\Bigl( \tau \in \cS (X), \mbox{$\gamma$-disconnected} \Bigr)
\Rightarrow \# ( \tau ) + \# ( \tau^{-1} \gamma ) \leq |X|+2.
\end{equation}
So the definition made for $\snc (X, \gamma )$ in Definition
\ref{def:2.3} takes here the following form:
\[
\snc (X, \gamma ) = \bigl\{ \tau \in \cS (X) \mid \mbox{$\tau$ is
$\gamma$-connected and } \# ( \tau ) + \# ( \tau^{-1} \gamma ) = |X|
\bigr\}
\]
\begin{equation}  \label{eqn:2.10}
\cup \, \bigl\{ \tau \in \cS (X) \mid \mbox{$\tau$ is
$\gamma$-disconnected and } \# ( \tau ) + \# ( \tau^{-1} \gamma ) =
|X|+2 \bigr\} .
\end{equation}

We next state the counterparts of Definition \ref{def:2.4} and of
Proposition \ref{prop:2.5} from the preceding subsection. Instead of
the crossing pattern (DC) from Definition \ref{def:2.4} we will now
have some ``annular'' crossing patterns (AC-1), (AC-2), (AC-3). In
order to describe them, it is useful to introduce the following
notation.

\begin{notation}  \label{def:2.7}
For every $y \in Y$ and $z \in Z$ we will denote by $\lambda_{y,z}$
the permutation of $X$ which fixes $y$ and $z$, and organizes $X
\setminus \{ y,z \}$ in a cycle in the following way:
\begin{equation} \label{eqn:2.11}
\lambda_{y,z} = ( \ \gamma (y), \gamma^2 (y), \ldots ,
\gamma^{|Y|-1} (y), \gamma (z), \gamma^2 (z), \ldots ,
\gamma^{|Z|-1} (z) \ ).
\end{equation}
The permutations $\lambda_{y,z}$ will be called in what follows {\bf
AC-test permutations} (because they are used in the annular crossing
patterns (AC-2) and (AC-3) from the next definition).
\end{notation}

\begin{defn}   \label{def:2.8}
$1^o$ We will say that a permutation $\tau \in \cS (X)$ is {\bf
compatible with} $\gamma$ if for every orbit $A$ of $\tau$ the
following two conditions are satisfied:

(i) $\tau \downarrow (A \cap Y) = \gamma \downarrow (A \cap Y)$,
$\tau \downarrow (A \cap Z) = \gamma \downarrow (A \cap Z)$.

(ii) There exists at most one element $a' \in A \cap Y$ such that
$\tau (a') \in Z$, and there exists at most one element $a'' \in A
\cap Z$ such that $\tau (a'') \in Y$.

$2^o$ Let $\tau$ be a permutation of $X$. We define three {\bf
annular crossing patterns} for $\tau$ with respect to $\gamma$, as
follows:

\begin{tabular}{lll}
\vline & {\bf (AC-1)} & There exist four distinct
            elements $a,b,c,d \in X$ such that                       \\
\vline &          & $\gamma \downarrow \{ a,b,c,d \} = (a,b,c,d)$
      and $\tau \downarrow \{ a,b,c,d \} = (a,c)(b,d)$.              \\
\vline &              &                                              \\
\vline & {\bf (AC-2)} & There exist five distinct elements
            $a,b,c,y,z \in X$ such that $y \in Y$, $z \in Z$,        \\
\vline &      & $\lambda_{y,z} \downarrow \{ a,b,c \} = (a,b,c)$
      and $\tau \downarrow \{ a,b,c,y,z \} = (a,c,b)(y,z)$.          \\
\vline &              &                                              \\
\vline & {\bf (AC-3)} & There exist six distinct elements
            $a,b,c,d,y,z \in X$ such that $y \in Y$, $z \in Z$,      \\
\vline &    & $\lambda_{y,z} \downarrow \{ a,b,c,d \} = (a,b,c,d)$
     and $\tau \downarrow \{ a,b,c,d,y,z \} = (a,c)(b,d)(y,z)$.
\end{tabular}
\end{defn}

\vspace{6pt}

\begin{prop}   \label{prop:2.9}
Consider the set of annular non-crossing permutations $\snc (X,
\gamma )$, as in Definition \ref{def:2.3}. For a permutation $\tau$
of $X$ we have the equivalence:
\begin{equation}\label{eqn:2.12}
\tau \in \snc (X, \gamma ) \ \Leftrightarrow \ \left\{
\begin{array}{l}
\mbox{$\tau$ is compatible with $\gamma$, and $\tau$ does not
                                                     satisfy}  \\
\mbox{any of the crossing patterns (AC-1), }                   \\
\mbox{(AC-2), (AC-3) with respect to $\gamma$.}
\end{array}  \right.
\end{equation}
\end{prop}

For a proof of Proposition \ref{prop:2.9}, see section 6 of
\cite{MN04}. Note that in \cite{MN04} it is the condition on the
right-hand side of (\ref{eqn:2.12}) which is taken as definition for
$\snc (X, \gamma )$.

The initials ``AC'' in (AC-1), (AC-2), (AC-3) stand for ``Annular
Crossing''. This comes from the fact that in order to draw
permutations in $\snc (X, \gamma )$ one starts by representing the
elements of $X$ as points on the boundary of an annulus. The
convention used in \cite{MN04} is that the elements of $Y$ are
represented on the outer circle of the annulus, clockwise and in the
order indicated by $\gamma \downarrow Y$; and the elements of $Z$
are represented on the inner circle of the annulus, counterclockwise
and in the order indicated by $\gamma \downarrow Z$. In terms of
pictures drawn in this annulus, the fact that a permutation $\tau$
of $X$ belongs to $\snc (X, \gamma )$ corresponds then to the
following. One can draw a closed contour for each of the cycles of
$\tau$, such that

(i) each of the contours does not self-intersect, and goes
clockwisely around the region it encloses;

(ii) the region enclosed by each of the contours is contained in the
annulus;

(iii) regions enclosed by different contours are mutually disjoint.

For an explanation of why the existence of a drawing satisfying
(i)--(iii) corresponds to the algebraic conditions stated on the
right-hand side of the equivalence (\ref{eqn:2.12}), see Remarks 3.8
and 3.9 in \cite{MN04}. An example of how such a drawing looks is
shown in Figure 3.

\begin{center}
\scalebox{.32}{\includegraphics{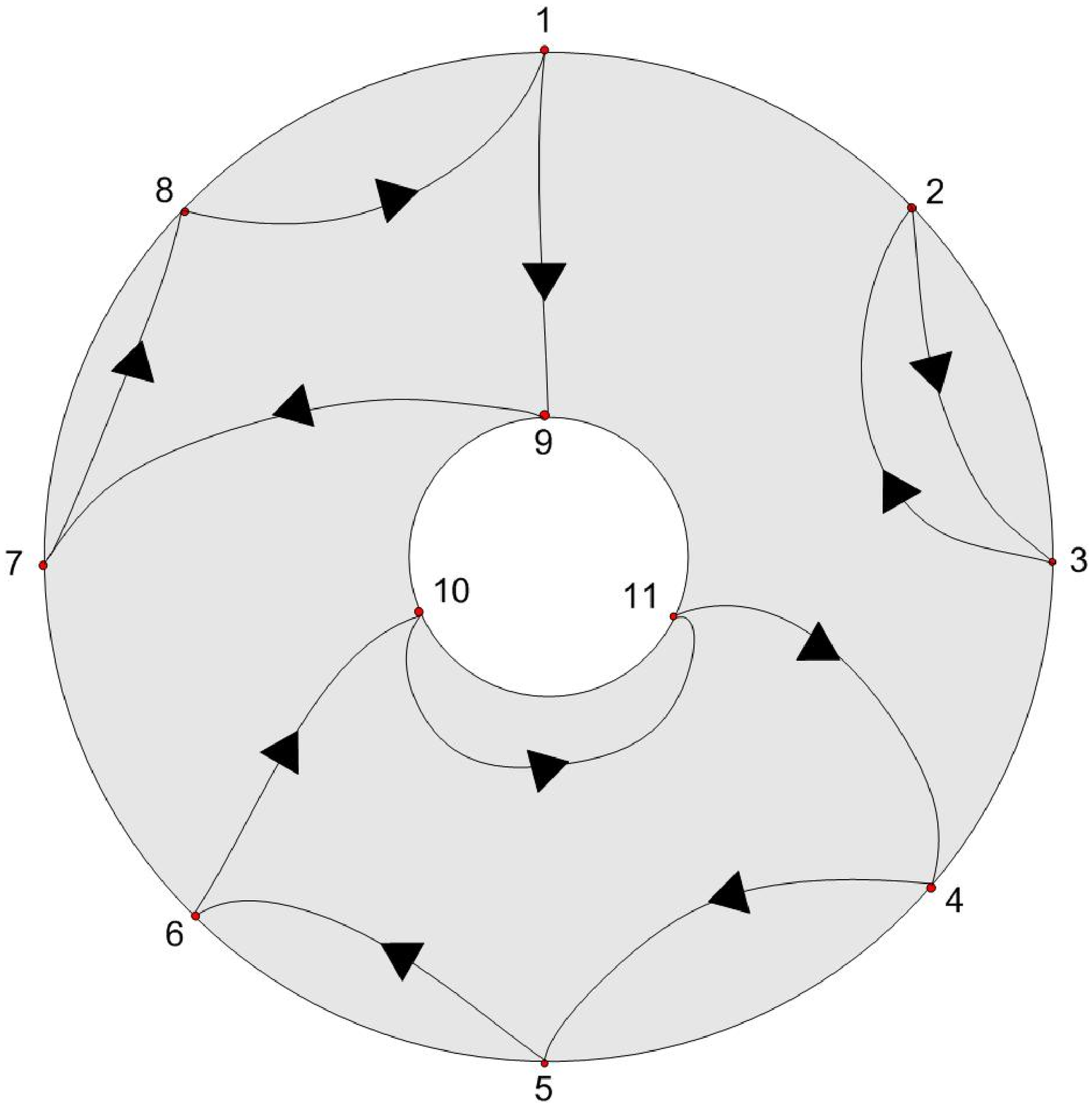}}

{\bf Figure 3.} An example of annular non-crossing permutation: $X=
\{ 1, \ldots , 11 \}$,

$\gamma = (1,2, \ldots , 8)(9,10,11), \ \tau =
(1,9,7,8)(2,3)(4,5,6,10,11) \in \cS_{nc} ( X, \gamma )$.
\end{center}

$\ $

\begin{center}
{\bf\large 3. \boldmath{$\sncb (p,q)$}, and proof of Theorem 1.1}
\end{center}
\setcounter{section}{3} \setcounter{equation}{0} \setcounter{thm}{0}
\setcounter{subsection}{0}

\noindent In this section we fix two positive integers $p$ and $q$.
We denote $p+q =:n$, and we put
\begin{equation}  \label{eqn:3.01}
X := \{ 1, \ldots , n \} \cup \{ -1, \ldots , -n \} .
\end{equation}
We consider the hyperoctahedral group $B_n = \{ \tau \in \cS (X)
\mid \tau (-i) = - \tau (i)$, $1 \leq i \leq n \}$, and the special
permutation
\begin{equation}  \label{eqn:3.02}
\gamma := \Bigl( 1, \ldots , p, -1, \ldots , -p \Bigr) \Bigl( p+1,
\ldots , n, -(p+1), \ldots , -n \Bigr) \in B_n.
\end{equation}
Following the notations from subsection 2.5, we will denote the
orbits of $\gamma$ by $Y$ and $Z$:
\begin{equation}  \label{eqn:3.03}
\left\{   \begin{array}{l}
Y := \{1, \ldots , p \} \cup \{ -1, \ldots , -p \}            \\
                                                              \\
Z := \{p+1, \ldots , n \} \cup \{ -(p+1), \ldots , -n \} .
\end{array}  \right.
\end{equation}

\begin{defn}  \label{def:3.1}
The set $\sncb (p,q)$ of annular non-crossing permutations of type B
is
\begin{equation}  \label{eqn:3.04}
\sncb (p,q) := \snc (X, \gamma ) \cap B_n,
\end{equation}
where $\snc (X, \gamma )$ is defined as in subsection 2.3 (see also
subsection 2.5).
\end{defn}

Our goal for the section is to prove that (as stated in Theorem
\ref{theorem1}) we have
\begin{equation}   \label{eqn:3.05}
\snc ( X, \gamma ) \cap B_n = \{ \tau \in B_n \mid \tau \leq  \gamma
\} ,
\end{equation}
where the partial order considered on $B_n$ is the one coming from
the length function $\lgb$. We will verify (\ref{eqn:3.05}) by
discussing separately the cases where we deal with
$\gamma$-connected and with $\gamma$-disconnected permutations of
$X$ (in Proposition \ref{prop:3.5} and in Proposition
\ref{prop:3.2}, respectively). We first deal with the
$\gamma$-disconnected case, which is immediately obtained from facts
known in the disc case.

\begin{prop}  \label{prop:3.2}
Consider the permutations induced by $\gamma$ on $Y$ and on $Z$:
\[
\alpha := \gamma \downarrow Y = \Bigl( 1, \ldots ,p,  -1, \ldots ,
-p \Bigr) , \ \ \beta  := \gamma \downarrow Z = \Bigl( p+1, \ldots
,n,  -(p+1), \ldots , -n \Bigr) .
\]

Given a $\gamma$-disconnected permutation $\tau \in B_n$, the
following three statements about $\tau$ are equivalent:

$(1) \ \  \tau \in \snc (X, \gamma )$.

$(2) \ \ \tau \downarrow Y \in \snc (Y, \alpha ) \mbox{ and } \tau
\downarrow Z \in \snc (Z, \beta )$.

$(3) \ \ \tau \leq \gamma$ with respect to the partial order
considered on $B_n$.
\end{prop}

\begin{proof}
The equivalence $(1) \Leftrightarrow (2)$ is proved in Remark 3.8 of
\cite{MN04}. For $(2) \Leftrightarrow (3)$, let $B_Y$ and $B_Z$
denote the Weyl groups of type B defined on $Y$ and respectively on
$Z$; that is, $B_Y$ consists of the permutations $\tau \in \cS (Y)$
which satisfy the condition $\tau (-i) = - \tau (i)$, $\forall \, i
\in Y$, and similarly for $B_Z$. Each of the groups $B_Y$ and $B_Z$
has a length function $\lgb$ on it, and a partial order defined by
starting from $\lgb$ (by the same recipe that was used to define the
partial order of $B_n$). It is immediately verified, directly from
definitions, that statement $(3)$ is equivalent to

\vspace{6pt}

$(3') \ \ \bigl( \tau \downarrow Y \bigr) \leq \alpha \mbox{ in }
B_Y, \mbox{ and } \bigl( \tau \downarrow Z \bigr)  \leq \beta \mbox{
in } B_Z.$

\vspace{6pt}

But $(3')$ is in turn equivalent to $(2)$, due to the result from
the disc case that was quoted in Equation (\ref{eqn:1.07}) of Remark
\ref{rem:1.6}.
\end{proof}

We now take on the $\gamma$-connected case. Here it comes in handy
to first record that a $\gamma$-connected permutation in $\sncb
(p,q)$ can never have inversion-invariant orbits. This fact can be
proved as follows.

\begin{lemma}  \label{lemma:3.3}
Let $\tau$ be a permutation in $\sncb (p,q)$. Then $\tau$ cannot
have a $\gamma$-connected orbit which is inversion-invariant.
\end{lemma}

\begin{proof} Assume for contradiction that $\tau$ has such an
orbit $A$. Since $A$ is $\gamma$-connected, we can find elements $i
\in A \cap Y$ and $j \in A \cap Z$ such that $\tau (i) = j$. But
then $-i$ also belongs to $A \cap Y$, and has $\tau (-i) = -j \in A
\cap Z$; so we see that $\tau$ does not satisfy the condition (ii)
in Definition \ref{def:2.8}.1 -- contradiction.
\end{proof}

\begin{prop} \label{prop:3.4}
 Let $\tau$ be a $\gamma$-connected permutation in $\sncb
(p,q)$. Then $\tau$ has no inversion-invariant orbits.
\end{prop}

\begin{proof}
By hypothesis, $\tau$ has a $\gamma$-connected orbit $C$. Let us fix
two elements $i,j \in C$ such that $i \in Y$, $j \in Z$, and $\tau
(i) =j$.

The preceding lemma implies that the orbit $-C$ of $\tau$ is
distinct from $C$. Note that we have $-i \in (-C) \cap Y$, $-j \in
(-C) \cap Z$, and $\tau (-i) = -j$.

Assume for contradiction that $\tau$ has an inversion-orbit $A$, and
let $k$ be an element of $A$. By looking at the six elements
$i,j,-i,-j,k,-k$ we see that $\tau$ satisfies the crossing pattern
(AC-3), contradiction.
\end{proof}

\begin{prop}  \label{prop:3.5}
Let $\tau \in B_n$ be a $\gamma$-connected permutation. Then we have
\begin{equation}  \label{eqn:3.06}
\tau \in \snc (X, \gamma ) \ \Leftrightarrow \ \Bigl( \tau \leq
\gamma \mbox{ in $B_n$}  \Bigr).
\end{equation}
\end{prop}

\begin{proof}
``$\Rightarrow$'' $\tau$ has no inversion-invariant orbits (by
Proposition \ref{prop:3.4}), so the formula (\ref{eqn:2.02}) for
length in $B_n$ gives us that
\begin{equation} \label{eqn:3.07}
\lgb (\tau) = n - \frac{1}{2} \# ( \tau ).
\end{equation}

Let us now look at the permutation $\tau^{-1} \gamma$. It is
immediate that this permutation is in $B_n$, and that it is
$\gamma$-connected (because $\tau$ is so). On the other hand it is
still true that $\tau^{-1} \gamma$ belongs to $\snc (X, \gamma )$ --
for a proof of this, see Corollary 6.5 of \cite{MN04}. Hence
$\tau^{-1} \gamma$ also is a $\gamma$-connected permutation in $\snc
(X, \gamma ) \cap B_n = \sncb (p,q)$, and we have the analogue of
Equation (\ref{eqn:3.07}), that
\begin{equation} \label{eqn:3.08}
\lgb (\tau^{-1} \gamma ) = n - \frac{1}{2} \# ( \tau^{-1} \gamma ).
\end{equation}

By adding together the Equations (\ref{eqn:3.07}) and
(\ref{eqn:3.08}), we obtain that
\[
\lgb (\tau) + \lgb (\tau^{-1} \gamma ) = 2n - \frac{1}{2} \Bigl( \#
( \tau ) + \# ( \tau^{-1} \gamma \Bigr).
\]
But we know that $\# ( \tau ) + \# ( \tau^{-1} \gamma ) = 2n$ (see
Equation (\ref{eqn:2.10})). Thus we have obtained precisely that
\[
\lgb (\tau) + \lgb (\tau^{-1} \gamma) = 2n - \frac{1}{2} (2n) = n =
\lgb ( \gamma ),
\]
and we conclude that $\tau \leq \gamma$.

``$\Leftarrow$'' In view of Equation (\ref{eqn:2.10}) it will
suffice to show that
\begin{equation}\label{eqn:3.09}
\# ( \tau ) + \# ( \tau^{-1} \gamma ) \geq 2n.
\end{equation}

Let $k$ and $l$ denote the number of inversion-invariant orbits of
the permutations $\tau$ and $\tau^{-1} \gamma$, respectively. Then
$\lgb (\tau) = n - ( \# ( \tau ) -k )/2$ and $\lgb (\tau^{-1}
\gamma) = n - ( \# ( \tau^{-1} \gamma) -l )/2$, so we get that
\[
n = \lgb ( \gamma ) = \lgb (\tau) + \lgb (\tau^{-1} \gamma) = 2n -
\frac{1}{2} \Bigl( \# ( \tau ) + \# ( \tau^{-1} \gamma ) -k -l
\Bigr).
\]
Hence $\# ( \tau ) + \# ( \tau^{-1} \gamma ) = 2n+k +l$, and
(\ref{eqn:3.09}) follows.
\end{proof}

$\ $

\begin{center}
{\bf\large 4. The map \boldmath{$\Otilda$} and the poset
\boldmath{$\ncb (p,q)$} }
\end{center}
\setcounter{section}{4} \setcounter{equation}{0} \setcounter{thm}{0}
\setcounter{subsection}{0}

\noindent Throughout this section we continue to use the notations
$p,q$, $n := p+q$, $X,Y,Z$, $\gamma$ from Section 3.

\subsection{Orbits of permutations from \boldmath{$\sncb (p,q)$} }

\begin{notation}   \label{notation4.1}
We will denote
\begin{equation}  \label{eqn:4.1}
\ob (p,q) := \left\{ A \subseteq X \begin{array}{cl}
\vline  & \exists \, \tau \in \sncb (p,q) \mbox{ such that} \\
\vline  & \mbox{$A$ is an orbit of $\tau$}
\end{array}  \right\} .
\end{equation}
\end{notation}

\begin{rem}   \label{remark4.2}
Let $A$ be a set in $\ob (p,q)$. A permutation in $\sncb (p,q)$
which has $A$ as an orbit must also have $-A$ as an orbit, and this
implies that either $A= -A$, or $A \cap (-A) = \emptyset$. In the
case when $A= -A$, we must have that $A \subseteq Y$ or $A \subseteq
Z$, because a permutation in $\sncb (p,q)$ which has an
inversion-invariant orbit must be $\gamma$-disconnected (see
Proposition \ref{prop:3.4}).
\end{rem}

\begin{lemma}   \label{lemma4.3}
$1^o$ Let $A \in \ob (p,q)$ be such that $A$ is
$\gamma$-disconnected (that is, we have $A \subseteq Y$ or $A
\subseteq Z$). Let $\tau \in \sncb (p,q)$ be such that $A$ is an
orbit of $\tau$. Then
\begin{equation}  \label{eqn:4.2}
\tau \induc A \ = \ \gamma \induc A.
\end{equation}

$2^o$ Let $A \in \ob (p,q)$ be such that $A$ is $\gamma$-connected
(that is, $A \cap Y \neq \emptyset \neq A \cap Z$). Let $\tau \in
\sncb (p,q)$ be such that $A$ is an orbit of $\tau$. On the other
hand consider two elements $y \in A \cap Y$ and $z \in A \cap Z$,
and look at the AC-test permutation $\lambda_{-y, -z} \in \cS (X)$
(defined as in Notation \ref{def:2.7}). Then
\begin{equation}  \label{eqn:4.3}
\tau \induc A \ = \ \lambda_{-y,-z} \induc A.
\end{equation}
\end{lemma}

\begin{proof} $1^o$ If $A \subseteq Y$, then
$\tau \induc A = ( \tau \induc Y ) \induc A$ = $( \gamma \induc
Y)\induc A = \gamma \induc A$ (we used the equality $\tau \induc Y =
\gamma \induc Y$, which is part of the requirements of compatibility
between $\tau$ and $\gamma$). The case when $A \subseteq Z$ is
analogous.

$2^o$ As observed in Remark \ref{remark4.2}, we have $A \cap (-A) =
\emptyset$.  So $-y, -z \not\in A$, which in turn implies that
$\lambda_{-y,-z} \induc A$ is a cyclic permutation of $A$.

If $|A| \leq 2$, then the equality (\ref{eqn:4.3}) follows just from
the fact that both $\lambda_{-y,-z} \induc A$ and $\tau \induc A$
are cyclic permutations of $A$.

Suppose then that $|A| \geq 3$. If the equality (\ref{eqn:4.3}) did
not hold, then there would exist three distinct elements $a,b,c \in
A$ such that
\[
\lambda_{-y,-z} \induc \{ a,b,c \} = (a,b,c), \ \ \tau \induc \{
a,b,c \} = (a,c,b).
\]
But then the five elements $a,b,c,-y,-z$ would produce an occurrence
of the crossing pattern (AC-2) in $\tau$ -- contradiction.
\end{proof}

\begin{defn} \label{definition4.4}
Let $A$ be a set in $\ob (p,q)$. From the preceding lemma it is
immediate that if $\tau_1 , \tau_2$ are permutations in $\sncb
(p,q)$ which have $A$ as an orbit, then we must have $\tau_1 \induc
A = \tau_2 \induc A$. It thus makes sense to define a permutation
$\mu_A \in \cS (A)$ by stipulating that
\begin{equation}  \label{eqn:4.4}
\mu_A = \tau \induc A,
\end{equation}
where $\tau$ is an arbitrary permutation in $\sncb (p,q)$ having $A$
as an orbit. We will refer to $\mu_A$ as the {\bf canonical
permutation} of $A$.
\end{defn}

\begin{rem}  \label{remark4.4}
Let $A$ be a set in $\ob (p,q)$, and consider the canonical
permutation $\mu_A \in \cS (A)$ defined above.

$1^o$ Equations (\ref{eqn:4.2}) and (\ref{eqn:4.3}) from Lemma
\ref{lemma4.3} give us ``explicit'' formulas for $\mu_A$: if $A$ is
$\gamma$-disconnected then
\begin{equation}  \label{eqn:4.5}
\mu_A = \gamma \induc A;
\end{equation}
while if $A$ is $\gamma$-connected (which implies that $A \cap (-A)
= \emptyset$) then
\begin{equation}  \label{eqn:4.6}
\mu_A = \lambda_{-y,-z} \induc A,
\end{equation}
for an arbitrary choice of $y \in A \cap Y$ and $z \in A \cap Z$.

$2^o$ Note that in the case when $A$ is $\gamma$-connected we still
have that
\begin{equation}  \label{eqn:4.7}
\mu_A \induc (A \cap Y) = \gamma \induc (A \cap Y), \ \ \mu_A \induc
(A \cap Z) = \gamma \induc (A \cap Z).
\end{equation}
The first of these two equalities follows from the immediate
observation that
\[
\lambda_{-y,-z} \induc (Y \setminus\{ -y \} ) = \gamma \induc (Y
\setminus\{ -y \} ),
\]
combined with the fact that $A \cap Y \subseteq Y \setminus \{ -y
\}$. The second equality is proved by a similar argument, this time
in reference to $A \cap Z$.

$3^o$ Let us record here a fact that will be used later: suppose
that $A$ is $\gamma$-connected and that $a,b,c,d$ are four distinct
elements of $A$ such that $\mu_A \downarrow \{ a,b,c,d \} =
(a,b,c,d)$. Then it is not possible to have $a,c \in Y$ and $b,d \in
Z$. Indeed, let us pick some elements $y \in A \cap Y$ and $z \in A
\cap Z$. From part $1^o$ of this remark it follows that
\[
\lambda_{-y,-z} \induc \{ a,b,c,d \} = \mu_A \induc \{ a,b,c,d \} =
(a,b,c,d);
\]
and it is clear, directly from the definition of $\lambda_{-y,-z}$
(see Notation \ref{def:2.7}), that $\lambda_{-y,-z} \induc \{
a,b,c,d \}$ could not be $(a,b,c,d)$ if we were to have $a,c \in Y$
and $b,d \in Z$.
\end{rem}

\subsection{The partitions \boldmath{$\Omega ( \tau )$} and
\boldmath{$\Otilda ( \tau )$} }

\begin{rem}   \label{rem:4.6}
In this subsection we move from individual orbits to {\em orbit
partitions} for permutations in $\sncb (p,q)$; that is, for every
$\tau \in \sncb (p,q)$ we consider the partitions $\Omega ( \tau )$
and $\Otilda ( \tau )$ defined in Notation \ref{def:1.2} of the
Introduction section. From the considerations in subsection 4.1 it
follows that the orbit map
\begin{equation}  \label{eqn:4.8}
\sncb (p,q) \ni \tau \mapsto \cO ( \tau ) \in \cP (X)
\end{equation}
is one-to-one; indeed, if $\tau \in \sncb (p,q)$ has orbit partition
$\pi \in \cP (X)$, then we know how to retrieve $\tau$ from $\pi$ --
we just have to put together the canonical permutations $\mu_A \in
\cS (A)$, where $A$ runs in the set of blocks of $\pi$. But let us
note that the orbit map (\ref{eqn:4.8}) is not order preserving,
when $\sncb (p,q)$ has the partial ordered induced from $B_n$, while
$\cP (X)$ is partially ordered by reverse refinement. Indeed, it is
clear for instance that if $\tau \in \sncb (p,q)$ is
$\gamma$-connected, then we have $\tau \leq \gamma$, but $\cO ( \tau
) \not\leq \cO ( \gamma )$. The next lemma shows that the
order-preservation issue is resolved if one works with $\Otilda (
\tau )$ instead of $\Omega ( \tau )$.
\end{rem}

\begin{lemma}   \label{lemma4.7}
The map $B_n \ni \tau \mapsto \Otilda ( \tau ) \in \cP (X)$ is order
preserving, where the partial order considered on $B_n$ is the one
coming from the length function $\lgb$, while $\cP (X)$ is partially
ordered by reverse refinement.
\end{lemma}

\begin{proof}
By using the explicit description of the cover relation in $B_n$ (as
reviewed in Proposition \ref{proposition3.5}), it is easily seen
that we have $\Otilda ( \sigma ) \leq \Otilda ( \tau )$ when $\sigma
, \tau \in B_n$ are such that $\tau$ covers $\sigma$. This in turn
immediately implies that the inequality $\Otilda ( \sigma ) \leq
\Otilda ( \tau )$ must actually hold whenever $\sigma \leq \tau$ in
$B_n$.
\end{proof}

On the other hand let us point out that if $\tau \in \sncb (p,q)$,
then going from $\cO ( \tau )$ to $\Otilda ( \tau )$ is only a minor
adjustment which can always be reversed, as explained in the next
lemma.

\begin{lemma}  \label{lemma4.8}
Let $\tau$ be a permutation in $\sncb (p,q)$, and consider the
following condition on the partition $\Otilda ( \tau)$:
\begin{equation}  \label{eqn:4.9}
\mbox{``There exists a block $A$ of $\Otilda ( \tau )$ which is
inversion-invariant and $\gamma$-connecting.''}
\end{equation}
If this condition is fulfilled, then the block $A$ with the deemed
properties $(A= -A$ and $A \cap Y \neq \emptyset \neq A \cap Z)$ is
uniquely determined, and $\cO ( \tau )$ is obtained from $\Otilda (
\tau )$ by splitting $A$ into $A \cap Y$ and $A \cap Z$. In the
opposite case, when the above condition is not fulfilled, we have
that $\cO ( \tau ) = \Otilda ( \tau )$.
\end{lemma}

\begin{proof} We discuss separately the cases when $\tau$ is
$\gamma$-connected and when it is $\gamma$-disconnected.

{\em Case 1.} If $\tau$ is $\gamma$-connected, then we know that
$\tau$ has no inversion-invariant orbits (by Proposition
\ref{prop:3.4}). In this case we observe that $\Otilda ( \tau ) =
\cO ( \tau )$, and that, on the other hand, $\Otilda ( \tau )$ does
not satisfy the condition (\ref{eqn:4.9}). The conclusion of the
lemma checks out.

{\em Case 2.} Suppose now that $\tau$ is $\gamma$-disconnected. Then
$\cO ( \tau )$ has at most two inversion-invariant orbits; and
moreover, if $\cO ( \tau )$ has exactly two inversion-invariant
orbits, then one of them is contained in $Y$ and the other is
contained in $Z$. This follows immediately from Proposition
\ref{prop:3.2}, and the fact that a permutation in $\sncb (p)$ or in
$\sncb (q)$ has at most one inversion-invariant orbit (the latter
fact is explained on p. 198 of \cite{R97}, the terminology used
there being that ``a non-crossing partition of type B has at most
one zero-block''). It is thus clear that the only possibility for
$\Otilda ( \tau ) \neq  \cO ( \tau )$ is when both $\tau \induc Y$
and $\tau \induc Z$ have inversion-invariant orbits. This also is
the only possibility for having $\Otilda ( \tau )$ satisfy the
condition (\ref{eqn:4.9}) -- hence the conclusion of the lemma
checks out in this case as well.
\end{proof}

\begin{prop}   \label{prop:4.9}
The map $\sncb (p,q) \ni \tau \mapsto \Otilda ( \tau ) \in \cP (X)$
is one-to-one, and it is order preserving (where $\sncb (p,q)$ is
partially ordered as an interval of $B_n$, while $\cP (X)$ is
partially ordered by reverse refinement.)
\end{prop}

\begin{proof} The ``order preserving'' part of the proposition is a
direct consequence of Lemma \ref{lemma4.7}. The ``one-to-one'' part
is also immediate: if $\sigma , \tau \in \sncb (p,q)$ are such that
$\Otilda ( \sigma ) = \Otilda ( \tau )$, then Lemma \ref{lemma4.8}
implies that $\cO ( \sigma ) = \cO ( \tau )$, and then the
injectivity observed in Remark \ref{rem:4.6} implies that $\sigma =
\tau$.
\end{proof}

\subsection{Proof of Theorem 1.4}

We will first prove several lemmas concerning the canonical
permutations $\mu_A$ ($A \in \ob (p,q)$) which were introduced in
Definition \ref{definition4.4}.

\begin{lemma}   \label{lemma4.12}
Let $A,B \in \ob (p,q)$ be such that $A \subseteq B$. Then $\mu_B
\induc A = \mu_A$.
\end{lemma}

\begin{proof} If $A$ is $\gamma$-disconnected, then both $\mu_A$ and
$\mu_B \induc A$ are equal to $\gamma \induc A$. Let us then assume
that $A$ is $\gamma$-connected, and let us pick two elements $y \in
A \cap Y$ and $z \in A \cap Z$. We have in particular that $y \in B
\cap Y$ and $z \in B \cap Z$, and it follows that both $\mu_A$ and
$\mu_B \induc A$ are equal to $\lambda_{-y,-z} \induc A$.
\end{proof}

\begin{lemma}   \label{lemma4.13}
Let $A$ be a set in $\ob (p,q)$, and suppose that $\sigma$ is a
permutation in $\sncb (p,q)$ such that $A$ is a union of orbits of
$\sigma$. Then $\sigma \induc A \in \snc (A, \mu_A )$.
\end{lemma}

\begin{proof} We will use the description of $\snc (A, \mu_A )$ in
terms of crossing pattern (DC), as reviewed in Definition
\ref{def:2.4} and Proposition \ref{prop:2.5}.

We first check that $\sigma \induc A$ is compatible with $\mu_A$.
This amounts to checking that for every orbit $B$ of $\sigma \induc
A$ we have
\begin{equation}  \label{eqn:4.131}
( \sigma \induc A )\induc B = \mu_A \induc B.
\end{equation}
But every orbit $B$ of $\sigma \induc A$ is in fact an orbit of
$\sigma$ (since it is given that $A$ is a union of orbits of
$\sigma$); thus $B \in \ob (p,q)$, and both sides of Equation
(\ref{eqn:4.131}) are equal to the canonical permutation $\mu_B$
(where on the right-hand side we invoke the preceding lemma).

We now go to proving that $\sigma \induc A$ cannot display the
crossing pattern (DC) with respect to $\mu_A$. Assume for
contradiction that there exist four distinct points $a,b,c,d \in A$
such that
\begin{equation}  \label{eqn:4.132}
\mu_A \induc \{ a,b,c,d \} = (a,b,c,d), \ \ ( \sigma \induc A )
\induc \{ a,b,c,d \} = (a,c)(b,d).
\end{equation}
We distinguish two cases.

\vspace{10pt}

{\em Case 1.} $\{ a,b,c,d \}$ is a $\gamma$-disconnected subset of
$X$; that is, we have that either $\{ a,b,c,d \} \subseteq Y$ or $\{
a,b,c,d \} \subseteq Z$.

In this case, Equation (\ref{eqn:4.7}) from Remark \ref{remark4.4}.2
implies that $\mu_A \induc \{ a,b,c,d \} = \gamma \induc \{ a,b,c,d
\}$. Thus the conditions in (\ref{eqn:4.132}) amount to
\[
\gamma \induc \{ a,b,c,d \} = (a,b,c,d), \ \ \sigma \induc \{
a,b,c,d \} = (a,c)(b,d),
\]
and this implies that $\sigma$ displays the crossing pattern (AC-1)
with respect to $\gamma$ -- contradiction.

\vspace{10pt}

{\em Case 2.} $\{ a,b,c,d \}$ is a $\gamma$-connected subset of $X$
(i.e. $\{ a,b,c,d \} \cap Y \neq \emptyset \neq \{ a,b,c,d \} \cap
Z$).

In this case we must have that at least one of the two sets $\{ a,c
\}$ and $\{ b,d \}$ is $\gamma$-connected. Indeed, if both $\{ a,c
\}$ and $\{ b,d \}$ were $\gamma$-disconnected, then it would follow
that either we have $a,c \in Y$ and $b,d \in Z$, or we have $a,c \in
Z$ and $b,d \in Y$; but this comes in contradiction with Remark
\ref{remark4.4}.3. In the remaining part of the proof we will assume
that $\{ b,d \}$ is $\gamma$-connected (the discussion based on the
assumption ``$\{ a,c \}$ is $\gamma$-connected'' would go in the
same way).

Let us next record that the six elements $a,b,c,d,-b,-d$ of $X$ are
distinct from each other. Indeed, we have that $a,b,c,d$ are
distinct elements of $A$, while $-b,-d$ are distinct elements of
$-A$, and Remark \ref{remark4.2} implies that $A \cap (-A) =
\emptyset$ (we use here the fact that $A$ is $\gamma$-connected,
which holds because $A \supseteq \{ b,d \}$).

From the second equality stated in (\ref{eqn:4.132}) and the fact
that $\sigma \in B_n$ it is immediate that
\[
\sigma \induc \{ a,b,c,d,-b,-d \} = (a,c)(b,d)(-b,-d),
\]
while on the other hand we see that
\begin{align*}
\lambda_{-b,-d} \induc \{ a,b,c,d \}
& = ( \lambda_{-b,-d} \induc A) \induc \{ a,b,c,d \}    \\
& = \mu_A \induc \{ a,b,c,d \} \ \ \mbox{ (by Equation
      (\ref{eqn:4.6}) in Remark \ref{remark4.4}.1)}     \\
& = (a,b,c,d).
\end{align*}
Hence $\sigma$ displays the crossing pattern (AC-3) with respect to
$\gamma$ -- contradiction.
\end{proof}

\begin{lemma}   \label{lemma4.14}
Let $B$ and $C$ be sets in $\ob (p,q)$ such that $B= -B \subseteq Y$
and $C= -C \subseteq Z$. We denote $B \cup C =: A$. Suppose that
$\sigma$ is a permutation in $\sncb (p,q)$ such that $A$ is a union
of orbits of $\sigma$. Then $\sigma \induc A \in \snc (A, \gamma
\induc A)$.
\end{lemma}

\begin{proof}
The permutation $\gamma \induc A$ has exactly two orbits, namely $B$
and $C$. We will prove that $\sigma \induc A \in  \snc (A , \gamma
\induc A )$ by using the description of $\snc (A, \gamma \induc A)$
in terms of annular crossing patterns, as reviewed in Definition
\ref{def:2.8} and Proposition \ref{prop:2.9}.

Let us first look at the verification that $\sigma \induc A$ is
compatible with $\gamma \induc A$. Here we have to check that every
orbit $U$ of $\sigma \induc A$ satisfies the conditions (i)+(ii) of
Definition \ref{def:2.8}.1, in the appropriate reformulation where
$Y$ and $Z$ are replaced by $B$ and $C$. And indeed, these
reformulated conditions (i)+(ii) are immediate consequences of the
corresponding conditions (i)+(ii) satisfied by $\sigma \in \snc (X,
\gamma )$, and where we use the same $U$. In order to illustrate
what happens, let us work out for instance the condition (i). In the
reformulation for $\sigma \induc A$, this condition has the form
\[
\mbox{``$( \sigma \induc A ) \induc (U \cap B ) = ( \gamma \induc A
) \induc (U \cap B )$'',}
\]
where $U$ is an orbit of $\sigma$ such that $U \subseteq A$. So we
are required to check that $\sigma$ and $\gamma$ induce the same
permutation on $U \cap B$. But the corresponding condition which we
know to be satisfied by $\sigma$ is that $\sigma \induc (U \cap Y )
= \gamma \induc (U \cap Y )$, and this does indeed imply that
$\sigma \induc (U \cap B ) = \gamma \induc (U \cap B )$, since $U
\cap Y \supseteq U \cap B$.

The verification that $\sigma \induc A$ does not display any of the
annular crossing patterns (AC-1), (AC-2), (AC-3) with respect to
$\gamma \induc A$ goes along the same lines as in the preceding
paragraph. That is, if $\sigma \induc A$ displayed a crossing
pattern (AC-$i$) with respect to $\gamma \induc A$ (where $1 \leq i
\leq 3$), then the same set of 4, 5 or 6 points of $A$ could be used
to infer that $\sigma$ displays the crossing pattern (AC-$i$) with
respect to $\gamma$. The straightforward verification of this fact
is left to the reader. We only note here that when treating the
crossing patterns (AC-2) and (AC-3) one has to take into account the
following simple observation: if $b \in B$, $c \in C$, and
$\lambda_{b,c} \in \cS (X)$ is the AC-test permutation defined as in
Equation (\ref{eqn:2.11}) of Notation \ref{def:2.7}, then the
counterpart of $\lambda_{b,c}$ in connection to $\snc (A, \gamma
\induc A )$ coincides with $\lambda_{b,c} \induc A$.
\end{proof}

\begin{prop}  \label{proposition4.15}
Let $\sigma , \tau \in \sncb (p,q)$ be such that $\Otilda ( \sigma )
\leq \Otilda ( \tau )$. Then $\sigma \leq \tau$ in $B_n$.
\end{prop}

\begin{proof}
We will distinguish three cases.

{\em Case 1.} Both $\sigma$ and $\tau$ are $\gamma$-disconnected.

In this case each of $\sigma$ and $\tau$ is completely determined by
its restrictions to $Y$ and to $Z$. Let $B_Y$ and $B_Z$ be the Weyl
groups of type B defined as in Proposition \ref{prop:3.2}. It is
immediate that the required inequality $\sigma \leq \tau$ in $B_n$
will follow if we can prove that $\sigma \induc Y \leq \tau \induc
Y$ in $B_Y$ and $\sigma \induc Z \leq \tau \induc Z$ in $B_Z$.

Now, from the hypothesis that $\Otilda ( \sigma ) \leq \Otilda (
\tau )$ it follows that $\cO ( \sigma \induc Y ) \leq \cO ( \tau
\induc Y )$, since the blocks of $\sigma \induc Y$ (respectively
$\tau \induc Y$) are obtained by intersecting the blocks of $\Otilda
( \sigma )$ (respectively the blocks of $\cO ( \tau )$) with $Y$.
But Proposition \ref{prop:3.2} gives us that $\sigma \induc Y , \tau
\induc Y \in B_Y \simeq B_p$; so if we know that $\cO ( \sigma
\induc Y ) \leq \cO ( \tau \induc Y )$, then we can invoke the poset
isomorphism reviewed in (\ref{eqn:1.08}) of subsection 1.2 to
conclude that $\sigma \induc Y \leq \tau \induc Y$ in $B_Y$. The
inequality $\sigma \induc Z \leq \tau \induc Z$ in $B_Z$ is obtained
in a similar manner.

\vspace{10pt}

{\em Case 2.} $\tau$ has no inversion-invariant orbits.

In this case $\sigma$ cannot have inversion-invariant orbits either.
We have $\Otilda ( \sigma ) = \cO ( \sigma )$ and $\Otilda ( \tau )
= \cO ( \tau )$, thus our hypothesis is that $\cO ( \sigma ) \leq
\cO ( \tau )$.

Let $A$ be an orbit of $\tau$. Then $A$ is a union of orbits of
$\sigma$, and Lemma \ref{lemma4.13} gives us that $\sigma \induc A
\in \snc ( A, \tau \induc A )$. Observe that
\[
( \sigma \induc A)^{-1} ( \tau \induc A) = ( \sigma^{-1} \tau )
\induc A,
\]
thus Equation (\ref{eqn:2.06}) from subsection 2.4 gives us that
\begin{equation}  \label{eqn:4.12}
\# ( \sigma \induc A) + \# (( \sigma^{-1} \tau ) \induc A) = 1 +
|A|.
\end{equation}

In Equation (\ref{eqn:4.12}) let us sum over all orbits $A$ of
$\tau$, where we take into account that every orbit of $\sigma$ is
contained in precisely one orbit of $\tau$, and that (consequently)
the same is true for every orbit of $\sigma^{-1} \tau$. We get
\begin{equation}  \label{eqn:4.13}
\# ( \sigma ) \ + \ \# ( \sigma^{-1} \tau ) \ = \ \# ( \tau ) + 2n.
\end{equation}

Finally, we convert Equation (\ref{eqn:4.13}) into a formula which
involves lengths in $B_n$. If a permutation $\phi \in B_n$ has no
inversion-invariant orbits, then the relation between the length
$\lgb ( \phi )$ and the number of cycles $\# ( \phi )$ is
\begin{equation}  \label{eqn:4.14}
\# ( \phi ) = 2( n - \lgb ( \phi )).
\end{equation}
This formula applies to each of $\sigma, \sigma^{-1} \tau$ and
$\tau$ (where in the case of $\sigma^{-1} \tau$, the absence of
inversion-invariant orbits follows from the inequality $\cO (
\sigma^{-1} \tau ) \leq \cO ( \tau )$). By substituting this into
(\ref{eqn:4.13}) we get precisely that $\lgb ( \sigma ) + \lgb (
\sigma^{-1} \tau )$ = $\lgb ( \tau)$, and the required inequality
$\sigma \leq \tau$ follows.

\vspace{10pt}

{\em Case 3.} $\sigma$ and $\tau$ are neither as in Case 1 nor as in
Case 2.

In this case $\tau$ must have inversion-invariant orbits (otherwise
Case 2 would apply). Proposition \ref{prop:3.4} thus implies that
$\tau$ is $\gamma$-disconnected. But then $\sigma$ has to be
$\gamma$-connected, otherwise Case 1 would apply. From the given
inequality $\Otilda ( \sigma ) \leq \Otilda ( \tau )$ and the fact
that $\sigma$ is $\gamma$-connected we next infer that the partition
$\Otilda ( \tau )$ is $\gamma$-connected.

In the preceding paragraph we saw that $\tau$ is
$\gamma$-disconnected, but the partition $\Otilda ( \tau )$ is
$\gamma$-connected. The only way this can happen is if $\tau$ has
exactly two inversion-invariant orbits, an orbit $B=-B \subseteq Y$
and an orbit $C=-C \subseteq Z$. Then, denoting $B \cup C =: A_o$,
we have that $A_o$ is the unique $\gamma$-connected block of
$\Otilda ( \tau )$ (while all the other blocks of $\Otilda ( \tau )$
are actual orbits of $\tau$, and each of them is either contained in
$Y$ or contained in $Z$). In the preceding paragraph we also saw
that $\sigma$ is $\gamma$-connected; note that, due to the
inequality $\Omega ( \sigma ) \leq \Otilda ( \tau )$, all the
$\gamma$-connected orbits of $\sigma$ must be contained in $A_o$.

We now start to count orbits of $\sigma$ and of $\sigma^{-1} \tau$,
in the same way as we did in Case 2. For every orbit $A$ of $\tau$
such that $A \neq B,C$ we have that $A$ is a union of orbits of
$\sigma$ and we can do exactly the same calculation as shown in Case
2. We obtain, analogously to Equation (\ref{eqn:4.12}) from Case 2,
that
\begin{equation}  \label{eqn:4.15}
\# ( \sigma \induc A) + \# (( \sigma^{-1} \tau ) \induc A) = 1+|A|,
\ \ \mbox{$\forall \, A$ orbit of $\tau$, $A \neq B,C$.}
\end{equation}

On the other hand, $A_o = B \cup C$ also is a union of orbits of
$\sigma$. Lemma \ref{lemma4.14} applies to this situation, and gives
us that $\sigma \induc A_o \in \snc (A_o, \gamma \induc A_o )$. It
is convenient to replace here $\gamma \induc A_o$ by $\tau \induc
A_o$ (the equality $\gamma \induc A_o = \tau \induc A_o$ is the
combination of the two equalities $\gamma \induc B = \tau \induc B$
and $\gamma \induc C = \tau \induc C$, which hold because $\tau$ is
compatible with $\gamma$, in the sense of Definition \ref{def:2.8}).
So we obtain that $\sigma \induc A_o \in \snc (A_o, \tau \induc A_o
)$, and the genus formula for $\sigma \induc A_o$ and $\tau \induc
A_o$ gives us that
\begin{equation}  \label{eqn:4.16}
\# ( \sigma \induc A_o)+ \#( ( \sigma^{-1} \tau ) \induc A_o)=|A_o|.
\end{equation}
(On the right-hand side of (\ref{eqn:4.16}) we used $|A_o|$ rather
than ``$|A_o|+2$'' because we know that $\sigma \induc A_o$ is $(
\tau \induc A_o )$-connected. The latter fact is in fact a
consequence of the fact $\sigma$ has $\gamma$-connected blocks which
are contained in $A_o$.)

Let us now sum in Equation (\ref{eqn:4.15}) over all the orbits $A
\neq B,C$ of $\tau$, and let us also add Equation (\ref{eqn:4.16})
to the result of that summation. We get (analogously to Equation
(\ref{eqn:4.13}) from Case 2) that
\begin{equation}  \label{eqn:4.17}
\# ( \sigma ) \ + \ \# ( \sigma^{-1} \tau ) \ = \ \Bigl( \# ( \tau )
-2  \Bigr) + 2n.
\end{equation}

Finally, we convert Equation (\ref{eqn:4.17}) into a formula which
involves lengths in $B_n$. We leave it as an exercise to the reader
to verify that the permutations $\sigma$ and $\sigma^{-1} \tau$ do
not have inversion-invariant orbits (the verification has only one
non-trivial point, namely the absence of inversion-invariant orbits
of $( \sigma^{-1} \tau ) \induc A_o$, which is obtained by applying
a ``re-denoted'' version of Proposition \ref{prop:3.4} to the
permutation $( \sigma^{-1} \tau ) \induc A_o \in \snc (A_o, \tau
\induc A_o)$). Hence the conversion from $\# ( \sigma )$ and $\#
(\sigma^{-1} \tau )$ to the lengths $\lgb ( \sigma )$ and $\lgb (
\sigma^{-1} \tau )$ is done via the same formula (\ref{eqn:4.14}) as
we used in Case 2. The permutation $\tau$ has two
inversion-invariant orbits, hence the formula used for $\tau$ has to
be
\[
\# ( \tau ) \ = \ 2 ( n - \lgb  ( \tau ) + 1 ).
\]
When we use these formulas in order to rewrite Equation
(\ref{eqn:4.17}) in terms of lengths in $B_n$, we get that $\lgb (
\sigma ) + \lgb ( \sigma^{-1} \tau )$ = $\lgb ( \tau )$, and the
required inequality $\sigma \leq \tau$ is obtained in this case as
well.
\end{proof}

Finally, it is clear that Theorem \ref{theorem2} now follows, when
we combine the statements of Proposition \ref{prop:4.9} and of
Proposition \ref{proposition4.15}.

$\ $

\begin{center}
{\bf\large 5. Intersection meets for partitions in \boldmath{$\ncb
(p,q)$} }
\end{center}
\setcounter{section}{5} \setcounter{equation}{0} \setcounter{thm}{0}
\setcounter{subsection}{0}

\noindent In this section we continue to use the framework and
notations from Sections 3 and 4.

We are dealing with $\ncb (p,q)$, which is a set of partitions of $X
= \{ 1, \ldots ,n \} \cup \{ -1, \ldots , -n \}$, for $n = p+q$. For
any partitions $\pi , \rho$ of $X$ we will use the notation $\pi
\wedge \rho$ to refer to the {\bf intersection meet} of $\pi$ and
$\rho$; that is, $\pi \wedge \rho$ is the partition of $X$ into
blocks of the form $A \cap B$ where $A$ is a block of $\pi$, $B$ is
a block of $\rho$, and $A \cap B \neq \emptyset$. It is immediate
that $\pi \wedge \rho$ is the meet (greatest common lower bound) for
$\pi$ and $\rho$ in the lattice $\cP (X)$ of all partitions of $X$.

In connection to the notation $\pi \wedge \rho$, we emphasize that
the implication
\[
\pi, \rho \in \ncb (p,q) \ \ \mbox{?} \Longrightarrow \mbox{?} \ \
\pi \wedge \rho \in \ncb (p,q)
\]
is not true in general. And in fact, while $\ncb (p,q)$ is always a
ranked poset with partial order given by reverse refinement, it
isn't generally true that $\ncb (p,q)$ is a lattice with respect to
this partial order. In the present section we look at the following
question: if $\pi, \rho \in \ncb (p,q)$ and if it is to be that $\pi
\wedge \rho \not\in \ncb (p,q)$, then how exactly can this happen?

\subsection{The case when \boldmath{$\pi \wedge \rho$} is
\boldmath{$\gamma$}-disconnected}

\begin{defn}  \label{definition5.1}
Let $\theta$ be a partition in $\ncb (p)$, and let $\omega$ be a
partition in $\ncb (q)$. We define a partition $\pi$ of $X$ which
will be denoted by ``$\Phi ( \theta , \omega )$'', and is described
as follows.

(i) Whenever $A$ is a block of $\theta$ such that $A \neq -A$, we
take $A$ to be a block of $\pi$.

(ii) Whenever $B$ is a block of $\omega$ such that $B \neq -B$, we
take $B'$ to be a block of $\pi$, where
\[
B' := \{ b+p \mid b \in B, \ b>0 \} \cup \{ b-p \mid b \in B, \ b<0
\} \subseteq \{ p+1, \ldots , n \} \cup \{ -(p+1), \ldots , -n \}
\subseteq X.
\]

(iii) Let $U \subseteq X$ be the union of all the blocks of $\pi$
considered in (i) and (ii) above. If $U \neq X$, then we take $X
\setminus U$ to be a block of $\pi$.
\end{defn}

\begin{rem}  \label{remark5.2}
Let $\theta , \omega$ and $\pi = \Phi ( \theta , \omega )$ be as
above.

$1^o$ It is clear that if $M$ is a block of $\pi$, then $-M$ is a
block of $\pi$ as well. It is also clear that $\pi$ can have at most
one inversion-invariant block $M$, namely the block $X \setminus U$
from (iii) of Definition \ref{definition5.1} (if it is the case that
$U \neq X$). A moment's thought shows that the construction of $\pi$
can be succinctly described as follows: ``Every block of $\theta$
and every block of $\omega$ is identified to a subset of $X$, in the
natural way; this gives a partition $\pi_o$ of $X$. Then $\pi$ is
obtained from $\pi_o$ by joining together all the
inversion-invariant blocks of $\pi_o$ (if such blocks exist) into
one block of $\pi$''.

$2^o$ Let $B_X$ and $B_Y$ be the Weyl groups of type B considered in
the proof of Proposition \ref{prop:3.2}, and let us also follow
Proposition \ref{prop:3.2} in denoting $\alpha := \gamma \induc Y
\in B_Y$ and $\beta  := \gamma \induc Z \in B_Z$. We then have
canonical identifications
\begin{equation}  \label{eqn:5.1}
\snc (Y, \alpha ) \cap B_Y = \sncb (p) \simeq \ncb (p) \mbox{ and }
\snc (Z, \beta ) \cap B_Z \simeq \sncb (q) \simeq \ncb (q)
\end{equation}
(where the isomorphisms $\sncb (p) \simeq \ncb (p)$ and $\sncb (q)
\simeq \ncb (q)$ are as in Equation (\ref{eqn:1.08}) of Remark
\ref{rem:1.6}). By using these canonical identifications, the
construction of the partition $\pi = \Phi ( \theta , \omega )$ can
also be described as follows. We identify $\theta$ and $\omega$ with
permutations from $\snc (Y, \alpha ) \cap B_Y$ and respectively from
$\snc (Z, \beta  ) \cap B_Z$, in the canonical way from
(\ref{eqn:5.1}). The two permutations so obtained (one of $Y$ and
one of $Z$) can be combined together into one permutation $\tau$ of
$X$; note that, by Proposition \ref{prop:3.2}, $\tau$ is a
$\gamma$-disconnected permutation in $\sncb (p,q)$. Then $\pi$ can
be defined as being the partition $\Otilda ( \tau )$ for this
particular $\tau \in \sncb (p,q)$.
\end{rem}

\begin{prop}  \label{proposition5.3}
$1^o$ For every $\theta \in \ncb (p)$ and $\omega \in \ncb (q)$, the
partition $\Phi ( \theta , \omega )$ defined above belongs to $\ncb
(p,q)$.

$2^o$ The map $\Phi : \ncb (p) \times \ncb (q) \to \ncb (p,q)$ is
injective, and its range-set can be described as $\{ \Otilda ( \tau
) \mid \tau \in \sncb (p,q), \ \tau \mbox{ is $\gamma$-disconnected}
\}$.
\end{prop}

\begin{proof} Part $1^o$ and the description of the range-set of
$\Phi$ in part $2^o$ follow from the description of $\Phi ( \theta ,
\omega )$ observed in Remark \ref{remark5.2}.2. The injectivity of
$\Phi$ is immediate from the description of $\Phi ( \theta , \omega
)$ given in Definition \ref{definition5.1}.
\end{proof}

\begin{cor}   \label{corollary5.4}
The subset $\{ \Otilda ( \tau ) \mid \tau \in \sncb (p,q), \ \tau
\mbox{ is $\gamma$-disconnected} \}$ of $\ncb (p,q)$ is closed under
the operation ``$\wedge$'' of intersection meet.
\end{cor}

\begin{proof} This is immediate from Proposition
\ref{proposition5.3} and the straightforward verification (made
directly from Definition \ref{definition5.1}) that we have $\Phi (
\theta , \omega ) \wedge \Phi ( \theta ' , \omega ' )$ = $\Phi (
\theta \wedge \theta ' , \omega \wedge \omega ')$, for every $\theta
, \theta ' \in \ncb (p)$ and every $\omega , \omega ' \in \ncb (q)$.
\end{proof}

\begin{cor}   \label{corollary5.5}
Let $\pi , \rho$ be in $\ncb (p,q)$, and let us denote $\pi \wedge
\rho =: \nu$. If $\nu$ has inversion-invariant blocks, then $\nu \in
\ncb (p,q)$.
\end{cor}

\begin{proof} Let $N$ be an inversion-invariant block of $\nu$,
and let us write $N = M \cap M'$ where $M$ is a block of $\pi$ and
$M'$ is a block of $\rho$. Then $M \cap (-M) \supseteq N$, hence $M
\cap (-M) \neq \emptyset$, and $M$ must be an inversion-invariant
block of $\pi$. Similarly, $M'$ must be an inversion-invariant block
of $\rho$. From Proposition \ref{prop:3.4} it follows that we must
have $\pi = \Otilda ( \tau )$ and $\rho = \Otilda ( \tau ')$ for
some $\gamma$-disconnected permutations $\tau ,  \tau ' \in \sncb
(p,q)$. But then Corollary \ref{corollary5.4} gives us that $\nu$
also is of the form $\Otilda ( \sigma )$ for some
$\gamma$-disconnected permutation $\sigma \in \sncb (p,q)$, and in
particular we find that $\nu \in \ncb (p,q)$.
\end{proof}

In the remaining part of this subsection we will prove another
statement going along the same lines as the above corollary, but
where the hypothesis on $\nu$ will be that it is
$\gamma$-disconnected. When doing that, it will come in handy to use
the following notation.

\vspace{10pt}

\begin{notation}  \label{definition5.6}
Let $\pi$ be a partition of $X$.

$1^o$ We will denote by $\Psi_1 ( \pi )$ the partition of $\{ 1,
\ldots , p \} \cup \{ -1, \ldots , -p \}$ into blocks of the form $A
= M \cap Y$, with $M$ a block of $\pi$ such that $M \cap Y \neq
\emptyset$.

$2^o$ We will denote by $\Psi_2 ( \pi )$ the partition of $\{ 1,
\ldots , q \} \cup \{ -1, \ldots , -q \}$ into blocks of the form
\[
B = \{ b-p \mid b \in M \cap Z, \ b>0 \} \cup \{ b+p \mid b \in M
\cap Z, \ b<0 \}
\]
with $M$ a block of $\pi$ such that $M \cap Z \neq \emptyset$.
\end{notation}

\begin{lemma}  \label{lemma5.7}
Let $\pi$ be a partition in $\ncb (p,q)$, and consider the
partitions $\theta := \Psi_1 ( \pi )$ and $\omega := \Psi_2 ( \pi )$
from the preceding notation. Then $\theta \in \ncb (p)$ and $\omega
\in \ncb (q)$.
\end{lemma}

\begin{proof} We denote by $\tau$ the unique permutation in
$\sncb (p,q)$ which has $\Otilda ( \tau ) = \pi$.

Assume for contradiction that $\theta \not\in \ncb (p)$. Then there
exist two distinct blocks $A$ and $A'$ of $\theta$ and elements $a,c
\in A$, $b,d \in A'$ such that $\alpha \induc \{ a,b,c,d \} =
(a,b,c,d)$, where
\[
\alpha := \gamma \induc Y = (1, \ldots , p, -1, \ldots , -p) \in \cS
(Y).
\]
The blocks $A$ and $A'$ can be written as $M \cap Y$ and
respectively $M' \cap Y$, where $M$ and $M'$ are two distinct blocks
of $\pi$. By using the fact that $\pi = \Otilda ( \tau )$, it is
easily seen that $\tau \induc \{ a,b,c,d \} = (a,c)(b,d)$. On the
other hand it is clear that
\[
\gamma \induc \{ a,b,c,d \} = \alpha \induc \{ a,b,c,d \} =
(a,b,c,d),
\]
and it follows that $\tau$ satisfies the crossing pattern (AC-1) --
contradiction.

The verification that $\omega \in \ncb (q)$ is made on the same
lines as shown for $\theta$ in the preceding paragraph.
\end{proof}

\begin{cor}   \label{corollary5.8}
Let $\pi , \rho$ be in $\ncb (p,q)$, and let us denote $\pi \wedge
\rho =: \nu$. If $\nu$ is $\gamma$-disconnected, then $\nu \in \ncb
(p,q)$.
\end{cor}

\begin{proof} We will assume that $\nu$ has no inversion-invariant
blocks (if it has such blocks, then we just invoke Corollary
\ref{corollary5.5}).

Consider the partitions $\Psi_1 ( \nu )$ and $\Psi_2 ( \nu )$; we
claim that $\Psi_1 ( \nu ) \in \ncb (p)$ and $\Psi_2 ( \nu ) \in
\ncb (q)$. Indeed, directly from how the maps $\Psi_1 ( \cdot )$ and
$\Psi_2 ( \cdot )$ are defined (see Notation \ref{definition5.6}) it
is immediately checked that
\[
\Psi_1 ( \nu ) = \Psi_1 ( \pi \wedge \rho ) = \Psi_1 ( \pi ) \wedge
\Psi_1 ( \rho ), \mbox{ and } \Psi_2 ( \nu ) = \Psi_2 ( \pi \wedge
\rho ) = \Psi_2 ( \pi ) \wedge \Psi_2 ( \rho ).
\]
But $\Psi_1 ( \pi ), \Psi_1 ( \rho ) \in \ncb (p)$ (by Lemma
\ref{lemma5.7}), and $\ncb (p)$ is closed under intersection meets,
hence $\Psi_1 ( \nu ) \in \ncb (p)$. A similar argument shows that
$\Psi_2 ( \nu ) \in \ncb (q)$.

Now let us look at the partition $\Phi ( \, \Psi_1 ( \nu ), \Psi_2 (
\nu ) \, )$. Note that $\Psi_1 ( \nu )$ and $\Psi_2 ( \nu )$ have no
inversion-invariant blocks (due to the hypothesis that $\nu$ has no
such blocks). The description of $\Phi ( \, \Psi_1 ( \nu ), \Psi_2 (
\nu ) \, )$ given in Remark \ref{remark5.2}.1 thus says that $\Phi (
\, \Psi_1 ( \nu ), \Psi_2 ( \nu ) \, )$ is simply obtained by
identifying the blocks of $\Psi_1 ( \nu )$ and of $\Psi_2 ( \nu )$
to subsets of $X$, in the natural way. But then it becomes clear
that $\Phi ( \, \Psi_1 ( \nu ), \Psi_2 ( \nu ) \, )$ is $\nu$
itself, and Proposition \ref{proposition5.3}.1 implies that $\nu \in
\ncb (p,q)$, as required.
\end{proof}

\subsection{The case when \boldmath{$\pi \wedge \rho$} is
\boldmath{$\gamma$}-connected}

\begin{lemma}  \label{lemma5.9}
Consider the collection of sets $\ob (p,q)$ introduced in subsection
4.1. Let $B$ be a set in $\ob (p,q)$ such that $B \cap (-B) =
\emptyset$, and let $A$ be a non-empty subset of $B$. Then $A \in
\ob (p,q)$.
\end{lemma}

\begin{proof} By the definition of $\ob (p,q)$, we can find a
permutation $\tau \in \sncb (p,q)$ such that $B$ is an orbit of
$\tau$. Then $-B$ is an orbit of $\tau$ as well. Let $\sigma$ be the
permutation of $X$ defined as follows:

\noindent (i) The sets $A$ and $-A$ are orbits of $\sigma$, and we
have $\sigma \induc \pm A = \tau \induc \pm A$.

\noindent (ii) Every element of $B \setminus A$ and every element of
$(-B) \setminus (-A)$ is a fixed point for $\sigma$.

\noindent (iii) On the set $X \setminus (B \cup (-B))$ (which is a
union of orbits of $\tau$) the permutation $\sigma$ acts exactly as
$\tau$ does.

We claim that $\sigma \in \sncb (p,q)$. Indeed, on the one hand it
is clear that $\sigma \in B_n$. On the other hand, the fact that
$\sigma \in \snc (X, \gamma )$ is easily verified by using the
description of $\snc (X, \gamma )$ in terms of annular crossing
patterns: the compatibility of $\sigma$ with $\gamma$ follows
immediately from the compatibility of $\tau$ with $\gamma$, and it
is also immediate that if $\sigma$ satisfies the crossing pattern
(AC-$i$) for some $1 \leq i \leq 3$ then $\tau$ would satisfy the
same crossing pattern, for the same set of points of $X$. (In the
verification of the latter fact one uses the obvious remark that
fixed points of permutations of $X$ can not be involved in any of
the crossing patterns (AC-1), (AC-2), (AC-3).)

So $\sigma \in \sncb (p,q)$ and $A$ is an orbit of $\sigma$, which
implies that $A \in \ob (p,q)$.
\end{proof}

\begin{prop}  \label{proposition5.10}
Let $\pi , \rho$ be in $\ncb (p,q)$, and let us denote $\pi \wedge
\rho =: \nu$. Suppose that $\nu$ is $\gamma$-connected, and has no
inversion-invariant block.

$1^o$ Every block $A$ of $\nu$ belongs to the collection of sets
$\ob (p,q)$, and we can thus talk about the canonical permutation
$\mu_A$ (introduced in Definition \ref{definition4.4}).

$2^o$ Let $\tau$ be the permutation of $X$ which is uniquely
determined by the requirements that $\Omega ( \tau ) = \nu$ and that
$\tau \induc A = \mu_A$, for every block $A$ of $\nu$. Then $\tau$
belongs to the group $B_n$, it is compatible with $\gamma$ (in the
sense of Definition \ref{def:2.8}.1), and does not display the
crossing patterns (AC-1) and (AC-2) (as described in Definition
\ref{def:2.8}.2).
\end{prop}

\begin{proof} $1^o$ Let $A$ be a block of $\nu$, and let us
write $A = B \cap C$ where $B$ is a block of $\pi$ and $C$ is a
block of $\rho$. It cannot happen that $B$ and $C$ are both
inversion-invariant (if $B=-B$ and $C=-C$ then it would follow that
$A = -A$, in contradiction to the hypotheses given on $\nu$). Assume
for instance that $B$ is not inversion-invariant.

Observe that $B \in \ob (p,q)$. Indeed, $B$ is a block of $\pi$, and
$\pi$ is of the form $\Otilda ( \phi )$ for some $\phi \in \sncb
(p,q)$. From the definition of $\Otilda ( \phi )$ it follows that
either $B$ is an orbit of $\phi$ or a union of orbits of $\phi$; but
the latter possibility is ruled out by the fact that $B \cap (-B) =
\emptyset$ (where we take into account how $\Otilda ( \phi )$ was
defined, in Notation \ref{def:1.2}). Hence $B$ is an orbit of
$\phi$, and hence $B \in \ob (p,q)$.

But then Lemma \ref{lemma5.9} applies to $B$ and $A$, and gives us
that $A \in \ob (p,q)$ as well.

$2^o$ The fact that $\tau \in B_n$ is immediate. It is also
immediate that $\tau$ satisfies the conditions of compatibility with
$\gamma$. Indeed, these conditions are actually defined for the
individual cycles of $\tau$; so they have to be fulfilled since (by
the definition of the canonical permutations $\mu_A$) every cycle of
$\tau$ is stolen from some permutation in $\sncb (p,q)$.

The proof that $\tau$ cannot satisfy (AC-1) and (AC-2) relies
essentially on the fact that the definition for each of these
crossing patterns involves elements from {\em only two} orbits of
$\tau$. We will show the proof for (AC-2), and leave the analogous
argument for (AC-1) as an exercise to the reader.

So let us assume for contradiction that $\tau$ displays the crossing
pattern (AC-2), hence that there exist five distinct points
$a,b,c,y,z \in X$, with $y \in Y$ and $z \in Z$, such that
\begin{equation}  \label{eqn:5.10}
\lambda_{y,z} \induc \{ a,b,c \} = (a,b,c) \ \mbox{ and } \ \tau
\induc \{ a,b,c,y,z \} = (a,c,b)(y,z).
\end{equation}

We claim that $\{ a,b,c \}$ must be a $\gamma$-connected subset of
$X$. Indeed, let $A$ be the orbit of $\tau$ which contains $\{ a,b,c
\}$. If it happened that $\{ a,b,c \} \subseteq Y$ or $\{ a,b,c \}
\subseteq Z$  then we would deduce that
\begin{align*}
\tau \induc \{ a,b,c \}
&= \mu_A \induc \{ a,b,c \} & \mbox{ (by definition of $\tau$)} \\
&= \gamma \induc \{ a,b,c \} & \mbox{ (by Eqn.(\ref{eqn:4.7})
                                 in Remark \ref{remark4.4}.2)}  \\
&= \lambda_{y,z} \induc \{ a,b,c \} & \mbox{ (directly from
                       the definition of $\lambda_{y,z}$),}
\end{align*}
in contradiction to what was assumed in (\ref{eqn:5.10}).

Now let $A$ be as above and let $A'$ denote the orbit of $\tau$
which contains $\{ y,z \}$. Then $A,A'$ are blocks of $\nu$, so we
can write $A= B \cap C$ and $A' = B' \cap C'$ where $B,B'$ are
blocks of $\pi$ and $C,C'$ are blocks of $\rho$. We have that either
$B \neq B'$ or $C \neq C'$ (in the opposite case it would follow
that $A=A'$, in contradiction to how $\tau$ acts on $\{ a,b,c,y,z
\}$). By swapping the roles of $\pi$ and $\rho$ if necessary, we
will assume that $B \neq B'$. Note that each of the two blocks $B$
and $B'$ of $\pi$ is $\gamma$-connected (since $B \supseteq \{ a,b,c
\}$ and $B' \supseteq \{ y,z \}$).

Let $\phi$ be the unique permutation in $\sncb (p,q)$ with the
property that $\Otilda ( \phi ) = \pi$. Observe that $\phi$ is
$\gamma$-connected; indeed, if $\phi$ was to be
$\gamma$-disconnected then (as seen directly from the definition of
$\Otilda$) the partition $\Otilda ( \phi )$ would have at most one
$\gamma$-connected block, while we know that $\pi$ has at least two
such blocks, namely $B$ and $B'$. From the fact that $\phi$ is
$\gamma$-connected we further infer that $\phi$ has no
inversion-invariant orbits (Proposition \ref{prop:3.4}). This
implies that $\cO ( \phi ) = \Otilda ( \phi ) = \pi$, and we can
therefore be certain that $B$ and $B'$ are orbits of $\phi$.

We next prove that $\phi \induc \{ a,b,c \} = (a,c,b)$. To this end
we consider the canonical permutation $\mu_B$ associated to the set
$B \in \ob (p,q)$ (see Definition \ref{definition4.4}) and we write:
\begin{align*}
\phi \induc \{ a,b,c \}
&= \mu_B \induc \{ a,b,c \} & \mbox{ (by definition of $\mu_B$)} \\
&= ( \mu_B \induc A) \induc \{ a,b,c \} & \mbox{ (because
             $B \supseteq A \supseteq \{ a,b,c \}$) }            \\
&= \mu_A \induc \{ a,b,c \} & \mbox{ (by Lemma \ref{lemma4.12})} \\
&= \tau  \induc \{ a,b,c \} & \mbox{ (by definition of $\tau$)}  \\
&= (a,c,b) & \mbox{ (by (\ref{eqn:5.10})).}
\end{align*}

We have thus found that $\phi$ has two distinct orbits $B$ and $B'$
such that $B \supseteq \{ a,b,c \}$, $B' \supseteq \{ y,z \}$, and
such that $\phi \induc \{ a,b,c \} = (a,c,b)$. It is then clear that
$\phi \induc \{ a,b,c,y,z \} = (a,c,b)(y,z)$; in conjunction with
the equality $\lambda_{y,z} \induc \{ a,b,c \} = (a,b,c)$ from
(\ref{eqn:5.10}) this shows that $\phi$ satisfies the crossing
pattern (AC-2) -- contradiction.
\end{proof}

\begin{rem}  \label{remark5.11}
At this moment we narrowed down quite a bit the possibilities for
how it can happen that $\pi , \rho \in \ncb (p,q)$, but $\nu := \pi
\wedge \rho \not\in \ncb (p,q)$: we must have that $\nu$ is
$\gamma$-connected and without inversion-invariant blocks (because
of Corollaries \ref{corollary5.5} and \ref{corollary5.8}), and the
permutation $\tau$ constructed in Proposition \ref{proposition5.10}
must display the crossing pattern (AC-3).

It is somewhat disappointing to see that if $p,q \geq 2$, this one
possibility that was left (with $\tau$ displaying the crossing
pattern (AC-3)) {\em can in fact occur}. This is immediately seen by
looking at the example where $\pi = \cO ( \sigma )$ and $\rho = \cO
( \tau )$ for
\begin{equation}  \label{eqn:5.2}
\left\{  \begin{array}{ccl}
\sigma & = &  (1,2,p+1,p+2)(-1,-2,-(p+1),-(p+2) ), \\
\tau   & = &  (1,-(p+2), p+1,-2)(-1, p+2 ,-(p+1), 2 ).
\end{array}  \right.
\end{equation}
In fact, if $p,q \geq 2$ then one can argue directly that $\ncb
(p,q)$ is not a lattice, in the following way: let $\sigma , \tau$
be as in (\ref{eqn:5.2}), and consider on the other hand the
permutations
\begin{equation}  \label{eqn:5.3}
\sigma_o = (1,p+1)(-1,-(p+1)), \ \ \tau_o   = (2,p+2)(-2,-(p+2))
\in B_n.
\end{equation}
We denote $\Omega ( \sigma ) = \pi$, $\Omega ( \tau ) = \rho$,
$\Omega ( \sigma_o ) = \pi_o$, $\Omega ( \tau_o ) = \rho_o$. It is
straightforward to check that $\pi , \rho , \pi_o, \rho_o$ all
belong to $\ncb (p,q)$, satisfy the inequalities $\pi_o \leq \pi$,
$\pi_o \leq \rho$, $\rho_o \leq \pi$, $\rho_o \leq \rho$, and yet
there is no partition $\nu \in \ncb (p,q)$ such that $\pi_o , \rho_o
\leq \nu \leq \pi , \rho$.

Figure 4 shows how the partitions and permutations of this example
look in the particular case when $p=q=2$. (The double-bracket
notation ``$((1,2,3,4))$'' is a short-hand for
``$(1,2,3,4)(-1,-2,-3,-4)$'', and the same convention is also used
for the other three permutations represented in this figure.)

$\ $

\begin{center}
\scalebox{0.6}{\includegraphics{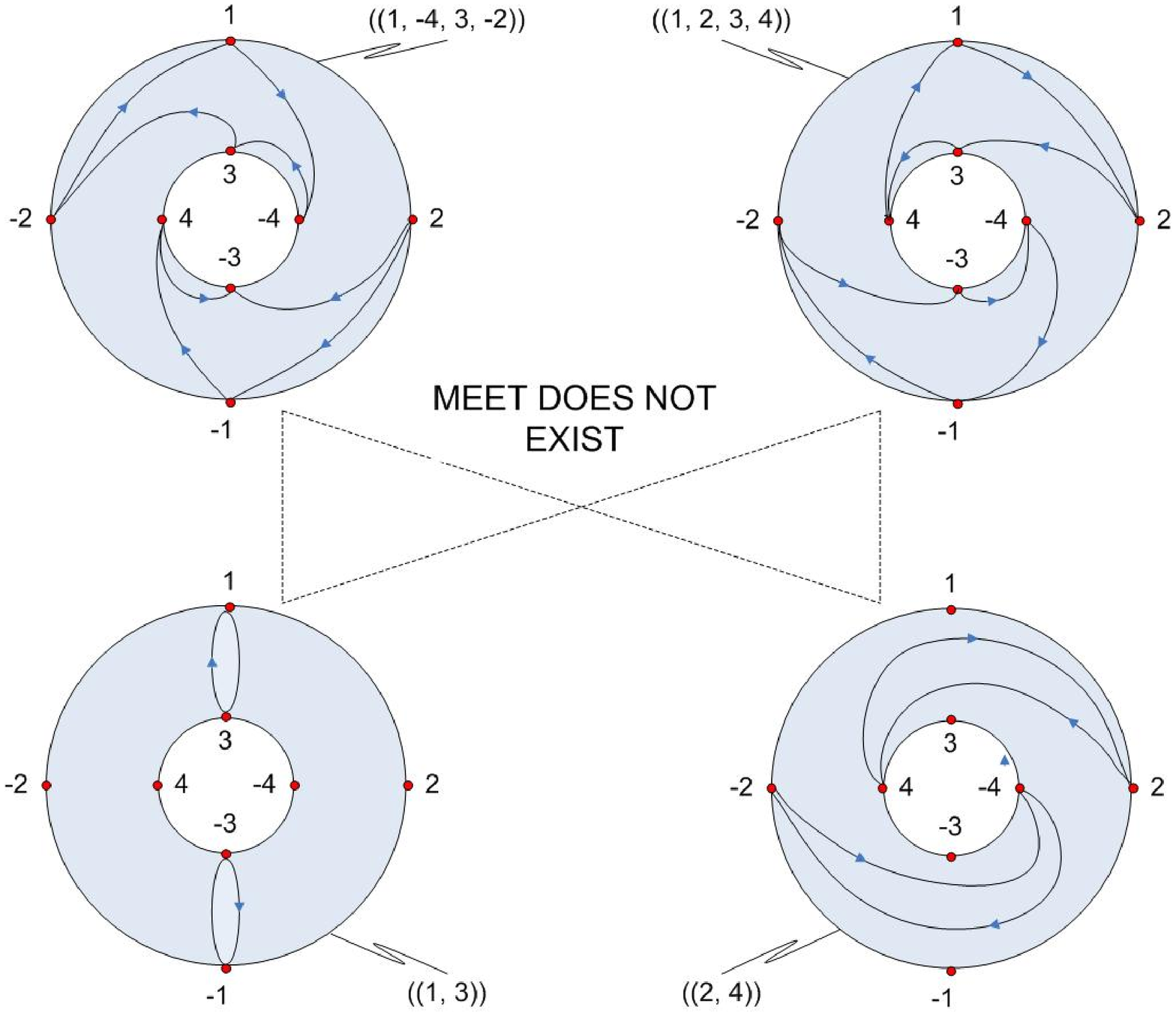}}

\vspace{6pt}

{\bf Figure 4.} Illustration for why $\ncb (2,2)$ is not a lattice.
\end{center}

$\ $

On the other hand, note that the above example takes advantage of
the existence of at least 4 points on each of the two circles of the
annulus. This detail really turns out to be essential -- in the next
section we will see that it is possible to ``finish the argument''
for the fact that $\pi \wedge \rho \in \ncb (p,q)$, if we place
ourselves in the particular situation when $p=n-1$ and $q=1$.
\end{rem}

$\ $

\begin{center}
{\bf\large 6. \boldmath{$\ncb (n-1,1)$} is a lattice}
\end{center}
\setcounter{section}{6} \setcounter{equation}{0} \setcounter{thm}{0}
\setcounter{subsection}{0}

\noindent This section is a continuation of Section 5, and inherits
all the notations used there ($X,Y,Z$, $\gamma , \ldots \ )$, with
the specification that the positive integers $p,q$ are now set to be
\begin{equation}     \label{eqn:6.1}
p= n-1, \ q=1, \ \ \mbox{ for some $n \geq 2$.}
\end{equation}
So the set $X$ continues to be $\{ 1,2, \ldots , n \} \cup \{ -1,-2,
\ldots , -n \}$, but $Y$ and $Z$ have now become
\[
Y = \{ 1,2, \ldots , n-1 \} \cup \{ -1,-2, \ldots , -(n-1) \} , \ \
Z = \{ n, -n \} ,
\]
$\gamma$ is the permutation
\[
\gamma = \Bigl( 1, \ldots , n-1, -1, \ldots , -(n-1) \Bigr) (n, -n)
\in B_n,
\]
and so on. Our goal for the section is to present the proof of
Theorem \ref{theorem3}, which states that $\ncb (n-1, 1)$ is a
lattice.

$\ $

\begin{rem}  \label{remark6.1}
It is easily seen that in order to prove Theorem \ref{theorem3}, all
we need to do is prove that $\ncb (n-1,1)$ is closed under the
operation ``$\wedge$'' of intersection meet which was reviewed at
the beginning of Section 5. Indeed, once this is established, it
becomes clear that every $\pi, \rho \in \ncb (n-1,1)$ have a
greatest common lower bound in $\ncb (n-1,1)$, which is precisely
their intersection meet; hence ``$\wedge$'' really gives a meet
operation on $\ncb (n-1,1)$. On the other hand it is obvious that
$\ncb (n-1,1)$ has a largest element, the partition of $X$ into only
one block; and it is easily checked that a finite poset with a meet
operation and which has a largest element has to be a lattice -- see
e.g. Proposition 3.3.1 in the monograph \cite{Sta96}.
\end{rem}

\begin{rem}    \label{remark6.2}
Let $\pi, \rho$ be two partitions in $\ncb (n-1,1)$, and consider
their intersection meet $\nu := \pi \wedge \rho$. Let us suppose
that $\nu$ is $\gamma$-connected and has no inversion-invariant
blocks, and let $\tau$ be the permutation of $X$ defined as in
Proposition \ref{proposition5.10} above: the orbit partition of
$\tau$ is equal to $\nu$, and for every block $A$ of $\nu$ we have
that $\tau \induc A = \mu_A$ (the canonical permutation of $A$
introduced in Definition \ref{definition4.4}). We will spend most
part of the present section by examining whether $\tau$ can display
the crossing pattern (AC-3), in order to eventually conclude that
this cannot happen.

So let us assume that $\tau$ does satisfy (AC-3), i.e. that there
exist six distinct elements $a,b,c,d,y,z \in X$ such that $y \in Y$,
$z \in Z$, and where we have
\begin{equation}  \label{eqn:6.2}
\lambda_{y,z} \induc \{ a,b,c,d \} = (a,b,c,d), \ \ \tau \induc \{
a,b,c,d,y,z \} = (a,c)(b,d)(y,z).
\end{equation}
In the current remark we make some observations about what this
entails, and we set some notations.

The main observation we want to record here is that {\em exactly one
of the sets $\{ a,c \}$ and $\{ b,d \}$ is $\gamma$-connected.}
Indeed, it is clear that $\{ a,c \}$ and $\{ b,d \}$ can't both be
$\gamma$-connected, as this would imply that among $a,b,c,d,y,z$
there are three distinct elements of $Z$ (namely $z$, one element
from $\{ a,c \} \cap Z$ and one from $\{ b,d \} \cap Z$); this is
not possible, since $Z= \{ n, -n \}$ only has two elements.

Suppose on the other hand that neither of $\{ a,c \}$ and $\{ b,d
\}$ are $\gamma$-connected, i.e. that each of them is either
contained in $Y$ or contained in $Z$. Note it is not possible to
have $\{ a,b,c,d \} \subseteq Y$ or $\{ a,b,c,d \} \subseteq Z$.
Indeed, if we had for instance that $\{ a,b,c,d \} \subseteq Y$ then
it would follow that
\begin{align*}
\lambda_{y,z} \induc \{ a,b,c,d \} & = \Bigl( \lambda_{y,z} \induc (
Y \setminus \{ y \} ) \Bigr)
                                          \induc \{ a,b,c,d \}  \\
& = \gamma \induc \{ a,b,c,d \} .
\end{align*}
This would lead to
\[
\gamma \induc \{ a,b,c,d \} = (a,b,c,d), \ \ \tau \induc \{ a,b,c,d
\} = (a,c)(b,d),
\]
and would imply that $\tau$ satisfies the crossing pattern (AC-1),
in contradiction to Proposition \ref{proposition5.10}. So if we
assume that $\{ a,c \}$ and $\{ b,d \}$ are both
$\gamma$-disconnected then it must follow that $\{ a,c \} \subseteq
Y$ and $\{ b,d \} \subseteq Z$ or vice-versa ($\{ a,c \} \subseteq
Z$ and $\{ b,d \} \subseteq Y$). But this situation can't occur
either, because, as explained in Remark \ref{remark4.4}.3, it is not
compatible with the assumption that $\lambda_{y,z} \induc  \{
a,b,c,d \} = (a,b,c,d)$.

Hence we know that exactly one of $\{ a,c \}$ and $\{ b,d \}$ is
$\gamma$-connected. By doing a circular permutation of $a,b,c,d$
(which does not affect the two equalities from (\ref{eqn:6.2})) we
may assume that $\{ a,c \}$ is $\gamma$-connected, and moreover,
that $a \in Z$ and $c \in Y$.

Now, $a$ and $z$ are distinct elements of $Z$; since $|Z|=2$, we
deduce that
\begin{equation}  \label{eqn:6.3}
a=-z, \ \ Z = \{ a,z \},
\end{equation}
and the remaining four elements $b,c,d,y$ that play a role in
(\ref{eqn:6.2}) all belong to $Y$. It is useful to also record here
that the cyclic order of $b,c,d,y$ on $Y$ is given by the formula
\begin{equation}  \label{eqn:6.35}
\gamma \induc \{ b,c,d,y \} = (b,c,d,y);
\end{equation}
this follows immediately by using the assumption (\ref{eqn:6.2})
that $\lambda_{y,z} \induc \{ a,b,c,d \} = (a,b,c,d)$, and by
checking how the long cycle of $\lambda_{y,z}$ goes, when one starts
at the point $a \in Z$.

In what follows we will denote by $A,A'$ and $A''$ the three
distinct orbits of $\tau$ (equivalently, blocks of $\nu$) which
contain $\{ a,c \}$, $\{ b,d \}$ and $\{ y,z \}$, respectively.
Since $\nu = \pi \wedge \rho$, we can write
\begin{equation}  \label{eqn:6.4}
A = B \cap C, \ A' = B' \cap C', \ A'' = B'' \cap C'',
\end{equation}
where $B,B',B''$ are blocks of $\pi$ and $C,C',C''$ are blocks of
$\rho$. Note that we have the relations
\begin{equation}  \label{eqn:6.5}
B'' = -B, \ C'' = -C,
\end{equation}
which hold because $B'' \ni z = -a \in -B$ and $C'' \ni z = -a \in
-C$.
\end{rem}

\begin{lemma}   \label{lemma6.3}
Consider the setting of the Remark \ref{remark6.2}.

$1^o$ It is not possible that any two of the three blocks $B,B',B''$
of $\pi$ are distinct from each other. Similarly, it is not possible
that any two of the three blocks $C,C',C''$ of $\rho$ are distinct
from each other.

$2^o$ It is not possible that $B=B'=B''$, and similarly, it is not
possible that $C=C'=C''$.
\end{lemma}

\begin{proof}
$1^o$ Assume for contradiction that $B,B'$ and $B''$ are three
distinct blocks of $\pi$. Let $\phi$ be the unique permutation in
$\sncb (n-1,1)$ with the property that $\Otilda ( \phi ) = \pi$.
Since $\pi$ has at least two distinct $\gamma$-connecting blocks
(namely $B$ and $B''$), we can use Lemma \ref{lemma4.8} to infer
that $\Otilda ( \phi )$ coincides in this case with the orbit
partition $\cO ( \phi )$. Hence $B,B',B''$ are three distinct orbits
of $\phi$, where $B \supseteq A \supseteq \{ a,c \}$, $B' \supseteq
A' \supseteq \{ b,d \}$, and $B'' \supseteq A'' \supseteq \{ y,z
\}$. It is then clear that
\[
\phi \induc \{ a,b,c,d,y,z \} = (a,c) (b,d) (y,z),
\]
and in conjunction with our standing assumption that $\lambda_{y,z}
\induc \{ a,b,c,d \} = (a,b,c,d)$ (made in Equation
(\ref{eqn:6.2})), this implies that $\phi$ satisfies the crossing
pattern (AC-3) -- contradiction.

The argument that $C,C',C''$ cannot be three distinct blocks of
$\rho$ is identical to the one shown above for $B,B',B''$.

$2^o$ If we had that $B=B'=B''$ then it would follow that $C,C',C''$
are three distinct blocks of $\rho$ (since the intersections $A = B
\cap C$, $A' = B' \cap C'$ and $A'' = B'' \cap C''$ give three
distinct orbits of $\tau$); but this is not possible, by part $1^o$
of the lemma. A similar argument rules out the possibility that
$C=C'=C''$.
\end{proof}

\begin{lemma}   \label{lemma6.4}
Consider the setting of the Remark \ref{remark6.2}. Then $B \neq
B''$ and $C \neq C''$.
\end{lemma}

\begin{proof} Assume for contradiction that $B=B''$. We observed
above (see (\ref{eqn:6.5})) that we also have $B''= -B$; hence $B$
is an inversion-invariant block of $\pi$. It is moreover clear that
$B$ is $\gamma$-connected, since $B \cap Y \ni c,y$ and $B \cap Z
\ni a,z$.

Let $\phi$ be the unique permutation in $\sncb (n-1,1)$ with the
property that $\Otilda ( \phi ) = \pi$. By Lemma \ref{lemma4.8}, $B$
is the unique block of $\pi$ which is both inversion-invariant and
$\gamma$-connected. The same lemma tells us that the partition $\cO
( \phi )$ of $X$ into orbits of $\phi$ consists of $B \cap Y$, $B
\cap Z$, and all the blocks of $\pi$ which are different from $B$.
Note in particular that $B'$ has to be an orbit of $\phi$ (indeed,
$B'$ is a block of $\pi$, and cannot be equal to $B=B''$, by part
$2^o$ of the preceding lemma).

But then let us look at the distinct orbits $B \cap Y$ and $B'$ of
$\phi$, and at the elements $c,y \in B \cap Y$ and $b,d \in B'$. All
these four elements belong to $Y$, and we have $\gamma \induc \{
b,c,d,y \} = (b,c,d,y)$ (see Equation (\ref{eqn:6.35}) above). This
leads us to the conclusion that $\phi$ satisfies the crossing
pattern (AC-1) -- contradiction.

So the assumption that $B=B''$ leads to contradiction, hence $B \neq
B''$. The proof that $C \neq C''$ is done in the same way.
\end{proof}

\begin{rem}  \label{remark6.5}
Consider the setting of the Remark \ref{remark6.2}. Due to the facts
proved in this setting in Lemmas \ref{lemma6.3} and \ref{lemma6.4},
we now know that the blocks $B,B',B''$ of $\pi$ are such that either
$B' = B$ or $B'=B''$ (indeed, Lemma \ref{lemma6.4} states that $B
\neq B''$, so having $B' \neq B$ and $B' \neq B''$ would contradict
Lemma \ref{lemma6.3}.1). Similarly, the blocks $C,C',C''$ of $\rho$
are such that either $C' = C$ or $C'=C''$.

Observe that it is not possible to have $B'=B$ and $C'=C$, because
$A = B \cap C$ and $A' = B' \cap C'$ are distinct orbits of the
permutation $\tau$. Similarly, it is not possible to have that
$B'=B''$ and $C'=C''$. So we are either in the case when $B'=B$,
$C'=C''$, or we are in the case when $B'=B''$, $C'=C$. By swapping,
if necessary, the roles of $\pi$ and of $\rho$ in the above
discussion, we can (and will) assume in what follows that it is the
first of these two cases which takes place.

So from now on we can continue our discussion by writing everything
in terms of the blocks $B$ and $C$. Indeed, the blocks $B',B''$ and
$C',C''$ that were introduced in (\ref{eqn:6.2}) can now be replaced
in terms of $B$ and $C$:
\begin{equation}  \label{eqn:6.6}
B'=B, \ B'' = -B, \ \  C'= C'' = -C.
\end{equation}
In terms of $B$ and $C$ alone, the statement of Lemma \ref{lemma6.4}
becomes that $B$ and $C$ are not inversion-invariant; hence we know
that
\begin{equation}  \label{eqn:6.7}
B \cap (-B) = \emptyset , \mbox{ and } C \cap (-C) = \emptyset .
\end{equation}
It is useful to also record here that (as an immediate consequence
of (\ref{eqn:6.6}) and of how $B,B',B''$ and $C,C',C''$ were defined
in Remark \ref{remark6.2}) we have
\begin{equation}  \label{eqn:6.8}
a,b,c,d,-y \in B, \ \ a,-b,c,-d,-y \in C.
\end{equation}
\end{rem}

\begin{prop}  \label{proposition6.6}
Let $\pi, \rho$ be two partitions in $\ncb (n-1,1)$, and consider
their intersection meet $\nu := \pi \wedge \rho$. Suppose that $\nu$
is $\gamma$-connected and has no inversion-invariant blocks, and let
$\tau$ be the permutation of $X$ defined as in Proposition
\ref{proposition5.10} above: the orbit partition of $\tau$ is equal
to $\nu$, and for every block $A$ of $\nu$ we have that $\tau \induc
A = \mu_A$ (the canonical permutation of $A$ introduced in
Definition \ref{definition4.4}). Then $\tau \in \sncb (n-1,1)$.
\end{prop}

\begin{proof}
The only thing to be proved about $\tau$ which was left out in
Proposition \ref{proposition5.10} is that it does not satisfy the
crossing pattern (AC-3). Assume for contradiction that $\tau$
satisfies (AC-3), and consider six distinct points $a,b,c,d,y,z \in
X$ with $y \in Y$ and $z \in Z$, such that the relations
(\ref{eqn:6.2}) from Remark \ref{remark6.2} are holding. The
arguments presented in Remark \ref{remark6.2}, in Lemmas
\ref{lemma6.3} and \ref{lemma6.4}, and in Remark \ref{remark6.5}
then tell us the following: at the cost of doing a cyclic
permutation of $a,b,c,d$ and of swapping if necessary the roles of
$\pi$ and $\rho$, we may assume that there exist a block $B$ of
$\pi$ and a block $C$ of $\rho$ such that (\ref{eqn:6.7}) and
(\ref{eqn:6.8}) hold. Moreover, the cyclic permutation we performed
on $a,b,c,d$ ensures that
\[
a=-z, \ \{ a,z \} =Z, \ \mbox{ and } \ b,c,d,y \in Y, \ \gamma
\induc \{ b,c,d,y \} = (b,c,d,y)
\]
(see Equations (\ref{eqn:6.3}) and (\ref{eqn:6.35}) in Remark
\ref{remark6.2}).

Let $\phi$ and $\psi$ be the permutations in $\sncb (n-1,1)$ which
have $\Otilda ( \phi ) = \pi$ and $\Otilda ( \psi ) = \rho$. Observe
that $B$ is an orbit of $\phi$. Indeed, the only way $B$ could be a
block of $\Otilda ( \phi )$ but not an orbit of $\phi$ would be if
$B$ was the union of two inversion-invariant orbits of $\phi$; but
this would imply that $B= -B$, and we know from (\ref{eqn:6.7}) that
$B \neq -B$. A similar argument shows that $C$ is an orbit of
$\psi$.

Let us next look at the elements $b,-b,c,d,y \in Y$. We claim that
these are five distinct elements of $Y$. Indeed, $b,c,d,y$ have to
be distinct because they are part of the set of six distinct
elements $a,b,c,d,y,z \in X$ that we started with. We next observe
that $-b$ is distinct from $b,c,d$ because $b,c,d \in B$, $-b \in
-B$ (by (\ref{eqn:6.8})), and $B \cap (-B) = \emptyset$ (by
(\ref{eqn:6.7})); a similar argument shows that $-b \neq y$ (we have
$-b \in C$, $y \in -C$, and $C \cap (-C) = \emptyset$).

We consider the cyclic permutation induced by $\gamma$ on the set
$\{ b,-b,c,d,y \}$. Since we know that $\gamma \induc \{ b,c,d,y \}
= (b,c,d,y)$, there are in fact only four possibilities for what
$\gamma \induc \{ b,-b,c,d,y \}$ can be. We group these four
possibilities into two cases, and we argue that each of the two
cases leads to contradiction.

\vspace{6pt}

{\em Case 1.} $\gamma \induc \{ b,-b,c,d,y \} = (b,-b,c,d,y)$, or
$\gamma \induc \{ b,-b,c,d,y \} = (b,c,-b,d,y)$.

In this case we have that $\gamma \induc \{ b,-b,d,y \} =
(b,-b,d,y)$, with $b,d \in B$ and $-b,y \in -B$. Since $B$ and $-B$
are two distinct orbits of $\phi$, it follows that $\tau \induc \{
b,-b,d,y \} = (b,d)(-b,y)$, and we find that $\phi$ satisfies the
crossing pattern (AC-1) -- contradiction.

\vspace{6pt}

{\em Case 2.} $\gamma \induc \{ b,-b,c,d,y \} = (b,c,d,-b,y)$, or
$\gamma \induc \{ b,-b,c,d,y \} = (b,c,d,y,-b)$.

In this case we have that $\gamma \induc \{ b,c,d,-b \} =
(b,c,d,-b)$, with $b,d \in -C$ and $c,-b \in C$. Since $C$ and $-C$
are two distinct orbits of $\psi$, it follows that  $\tau \induc \{
b,c,d,-b \} = (b,d)(c,-b)$, and we find that $\psi$ satisfies the
crossing pattern (AC-1) -- contradiction.
\end{proof}

\begin{cor}   \label{corollary6.7}
If $\pi , \rho$ are two partitions in $\ncb (n-1,1)$, then the
intersection meet $\pi \wedge \rho$ also belongs to $\ncb (n-1,1)$.
\end{cor}

\begin{proof} This follows immediately when the statement of
Proposition \ref{proposition6.6} is added to the discussion made in
Remark \ref{remark5.11} at the end of the preceding section.
\end{proof}

Finally, Theorem \ref{theorem3} follows from Corollary
\ref{corollary6.7}, in the way observed in the above Remark
\ref{remark6.1}.

$\ $

\begin{center}
{\bf\large 7. The case of type D}
\end{center}
\setcounter{section}{7} \setcounter{equation}{0} \setcounter{thm}{0}
\setcounter{subsection}{0}

\noindent The results of the paper were stated in the introduction
in the framework of the groups $B_n$, but all three Theorems
\ref{theorem1}, \ref{theorem2} and \ref{theorem3} have counterparts
that hold in the framework of the Weyl groups $D_n$. In this section
we present these counterparts of type D.

We will use the notations $p, q$, $n := p+q$, $X,Y, \gamma$ that
were introduced in Section 3. The Weyl group $D_n$ is the subgroup
of $\cS (X)$ defined as
\[
D_n = \Bigl\{ \tau \in \cS (X)
\begin{array}{cc}
\vline & \tau (-i) = - \tau (i), \ \forall \, i \in X, \mbox{ and} \\
\vline & \mbox{$\tau$ is an even permutation}
\end{array}  \Bigr\} .
\]
(Thus $D_n$ is a subgroup of index 2 of $B_n$.) The analogue of type
D for the set of annular non-crossing permutations $\sncb (p,q)$
from Definition \ref{def:3.1} is
\begin{equation}  \label{eqn:1.153}
\sncd (p,q) := \snc (X, \gamma ) \cap D_n .
\end{equation}

On the other hand we use on $D_n$ a length function $\lgd$, which is
defined with respect to the following set of generators of $D_n$:
\begin{equation}  \label{eqn:1.151}
\{ (i,j)(-i,-j) \mid 1 \leq i<j \leq n \} \cup \{ (i,-j)(-i,j) \mid
1 \leq i<j \leq n \} .
\end{equation}
That is, for every $\tau \in D_n$ we have that $\lgd ( \tau )$ is
the smallest possible $k$ such that $\tau$ can be factored as a
product of $k$ generators from (\ref{eqn:1.151}). The length
function $\lgd$ then defines a partial order on $D_n$, by the same
kind of formula as used in type B: for $\sigma , \tau \in D_n$ we
put
\begin{equation}  \label{eqn:1.152}
\sigma \leq \tau \ \ecdef \ \lgd ( \tau ) = \lgd ( \sigma ) + \lgd (
\sigma^{-1} \tau ).
\end{equation}

Now, the counterpart of type D for Theorem \ref{theorem1} turns out
to follow easily from the theorem itself, due to the following
easily checked observation about length functions: the length
function $\lgd$ on $D_n$ is in fact {\em the restriction} to $D_n$
of the length function of type B, $\lgb$ on $B_n$. This in turn
implies that for $\sigma , \tau \in D_n$ we have the equivalence
\begin{equation}  \label{eqn:1.154}
\Bigl( \sigma \leq \tau \mbox{ in } D_n \Bigr) \ \Leftrightarrow
\Bigl( \sigma \leq \tau \mbox{ in } B_n \Bigr) .
\end{equation}
But then we immediately get that:

\begin{cor}  \label{cor:7.1}
$\sncd (p,q) = \{ \tau \in D_n \mid \tau \leq \gamma \}$.
\end{cor}

\begin{proof}
We have that
\begin{align*}
\{ \tau \in D_n \mid \tau \leq \gamma \} & = \{ \tau \in B_n \mid
\tau \leq \gamma \} \cap D_n
       & \mbox{ (because of (\ref{eqn:1.154})) }             \\
& = \sncb (p,q) \cap D_n  & \mbox{ (by Theorem \ref{theorem1})} \\
& = ( \snc (X, \gamma ) \cap B_n) \cap D_n  & \mbox{ (by
                                definition of $\sncb (p,q)$) }    \\
& = \snc (X, \gamma ) \cap D_n  &                                 \\
& = \sncd (p,q)  & \mbox{ (by definition of $\sncd (p,q)$). }
\end{align*}
\end{proof}

Similarly, the counterpart of type D for Theorem \ref{theorem2} is a
corollary of Theorem \ref{theorem2}.

\begin{cor} \label{cor:7.2}
Let us denote
\begin{equation}  \label{eqn:1.155}
\ncd (p,q) := \{ \Otilda ( \tau ) \mid \tau \in \sncd (p,q) \}.
\end{equation}
Then the map
\begin{equation}  \label{eqn:1.156}
\sncd (p,q) \ni \tau \mapsto \Otilda ( \tau ) \in \ncd (p,q)
\end{equation}
is a poset isomorphism, where $\sncd (p,q)$ is partially ordered as
an interval of $D_n$, while $\ncd (p,q)$ is partially ordered by
reverse refinement.
\end{cor}

\begin{proof}
From the equivalence (\ref{eqn:1.154}) it follows that the partial
order considered on $\sncd (p,q)$ is the one induced from $\sncb
(p,q)$. On the other hand it is clear that the partial order on
$\ncd (p,q)$ is the one induced from $\ncb (p,q)$ (since for $\pi ,
\rho \in \ncd (p,q)$ the inequality ``$\pi \leq \rho$'' means that
every block of $\rho$ is a union of blocks of $\pi$, and this is
independent of whether $\pi , \rho$ are viewed as elements of $\ncd
(p,q)$ or as elements of $\ncb (p,q)$). But then the fact that in
(\ref{eqn:1.156}) we have a poset isomorphism follows by
appropriately restricting the poset isomorphism (\ref{eqn:1.06})
from Theorem \ref{theorem2}.
\end{proof}

Finally, let us discuss the counterpart of type D for Theorem
\ref{theorem3}. This does hold, that is, $\ncd (n-1,1)$ is a lattice
with respect to the partial order given by reverse refinement. But
this is not an immediate corollary of Theorem \ref{theorem3}.
Indeed, $\ncd (n-1,1)$ is a {\em subposet} of $\ncb (n-1,1)$, but is
not a {\em sublattice} of $\ncb (n-1,1)$ -- for $\pi , \rho \in \ncd
(n-1,1)$, the meet of $\pi$ and $\rho$ in $\ncd (n-1,1)$ doesn't
generally coincide with the ``intersection meet'' $\pi \wedge \rho$
described in Theorem \ref{theorem3} ! So here a different kind of
argument is required; but we are fortunate that we only need to
invoke the work previously done by Athanasiadis and Reiner in the
paper \cite{AR04}.

$\ $

\begin{rem}  \label{remark1.6}
For $n \geq 2$, the poset $\ncd (n-1,1)$ coincides exactly with the
poset constructed in \cite{AR04}, and denoted there as ``$NC^{(D)}
(n)$''. Thus $\ncd (n-1,1)$ is a lattice, by Proposition 3.1 of
\cite{AR04}.
\end{rem}

\vspace{10pt}

The annular interpretation for the lattice $NC^{(D)} (n)$ of
Athanasiadis and Reiner was observed independently by Krattenthaler
and M\"uller in Section 7 of their recent paper \cite{KM07}.

We conclude by pointing out a couple of clues that have to be
followed in order to make the connection between the poset $NC^{(D)}
(n)$ from \cite{AR04} and the poset $\ncd (n-1,1)$ of this paper.
The construction made in \cite{AR04} goes by drawing $1,2, \ldots ,$
$n-1, -1, -2, \ldots , -(n-1)$ around a circle, and by placing both
$n$ and $-n$ at the center of the circle. But if instead of putting
$n$ and $-n$ right at the center we put them on a small circle
concentric with the one containing $\pm 1, \pm 2 , \ldots , \pm
(n-1)$, then the partitions considered in the definition of
$NC^{(D)} (n)$ (see beginning of Section 3 in \cite{AR04}) become
{\em annular} non-crossing. Another point in \cite{AR04} which looks
puzzling at first sight is that if a partition $\pi \in NC^{(D)}
(n)$ has a zero-block (a block $B$ such that $B= -B$), then $\pm n$
are forced to belong to that block. But this corresponds exactly to
the passage from $\Omega ( \tau )$ to $\Otilda ( \tau )$ in Notation
\ref{def:1.2}. Indeed, if a permutation $\tau \in \sncd (n-1,1)$ has
inversion-invariant orbits, then it turns out that $\tau$ must have
exactly two such orbits, $M$ and $N$, where $M \subseteq \{ 1,
\ldots , n-1 \} \cup \{ -1, \ldots , -(n-1) \}$ and $N$ is forced to
be $\{ n, -n \}$; so the partition $\Otilda ( \tau )$ has exactly
one inversion-invariant block, $M \cup N$, which is forced to
contain $\pm n$.

$\ $

$\ $

{\bf Acknowledgements.} The starting point for the present paper was
a careful reading of the construction made by Athanasiadis and
Reiner in \cite{AR04}, and the observation that this construction
actually produces annular non-crossing partitions. The first named
author acknowledges his participation to a workshop on ``Braid
groups, clusters and free probability'' organized by the American
Institute of Mathematics in 2005, where he became acquainted to the
results in \cite{AR04}. We also acknowledge many enlightening
discussions that both authors had with Ian Goulden at various stages
of preparation of the paper, and several useful comments and
suggestions made by the referees for the paper.

$\ $

$\ $

Alexandru Nica: University of Waterloo.

Address: Department of Pure Mathematics, University of Waterloo,

Waterloo, Ontario N2L 3G1, Canada.

Email: anica@math.uwaterloo.ca

$\ $

Ion Oancea: University of Waterloo.

Address: Department of Pure Mathematics, University of Waterloo,

Waterloo, Ontario N2L 3G1, Canada.

Email: ioancea@math.uwaterloo.ca

\end{document}